\documentclass[10pt,reqno]{amsart}
\usepackage{amssymb,amsfonts,amsthm,amscd,stmaryrd,dsfont,esint,upgreek,constants,todonotes,amsmath}
\usepackage[mathcal]{euscript}
\usepackage{hyperref}
\usepackage[latin1]{inputenc}   
\usepackage[mathcal]{euscript}
\usepackage{graphicx,color}
\numberwithin{equation}{section}

\newcommand{\R}{\mathbb R}
\def\E{\mathbb E}
\def\P{\mathbb P}

\newconstantfamily{C}{symbol=C}
\newconstantfamily{c}{symbol=c}
\newconstantfamily{O}{symbol=\Omega}

\def\XXint#1#2#3{{\setbox0=\hbox{$#1{#2#3}{\int}$}
\vcenter{\hbox{$#2#3$}}\kern-.5\wd0}}

\numberwithin{equation}{section}
\newtheorem{thm}{Theorem}[section]
\newtheorem{lem}[thm]{Lemma}

\newtheorem{prop}[thm]{Proposition}

\theoremstyle{definition}

\def\smallnegint{\mathop{\int\mkern-13mu
        \raise.5ex\hbox{${\scriptscriptstyle\diagup}$}}\nolimits}
\def\ds{\displaystyle}

\def\ep{\varepsilon}

\def\oo{\infty}
\def\ol{\overline}
\def\ssetminus{\,\raise.4ex\hbox{$\scriptstyle\setminus$}\,}

\newcommand{\be}{\begin{equation}}
\newcommand{\ee}{\end{equation}}

\newcommand{\bc}{\begin{cases}}
\newcommand{\ec}{\end{cases}}
\newcommand{\bs}{\begin{split}}
\newcommand{\es}{\end{split}}
\newcommand{\pt}{\partial}

\newcommand{\vs}{\vskip.075in}

\renewcommand{\bar}{\overline}

\renewcommand{\hat}{\widehat}

\def\rt{\R\times(0,1)}

\def\bee{\begin{equation*}}
\def\eee{\end{equation*}}

\begin{document}


\title
[A convergence result for a planing problem for mean field games and a rigorous proof of a Ventcel-Freidlin-type Large Deviations Principle  for the $1+1$ KPZ equation]
{A convergence result for a planning problem for mean field games
and rigorous proof of a Ventcel-Freidlin-type Large Deviations Principle  for the $1+1$ KPZ equation}

\author[Pierre-Louis Lions  and Panagiotis E. Souganidis]{Pierre-Louis Lions and Panagiotis E. Souganidis}
\address{College de France, 11 Pl. Marcelin Berthelot, 75231 Paris, France}
\address{Department of Mathematics, University of Chicago, Chicago, Illinois 60637, USA}
\vskip-0.5in 

\dedicatory{Version: \today}

%
%
%
%
%
%
%


\begin{abstract} We study an innovative connection between a very singular  optimal transport variational problem with a congestion term connecting two Dirac measures, which is related to the planing problem in mean field games or gases with negative pressure, and a large deviation principle for the KPZ equation. The main result proves a  limit of a singular approixmate problem with vanishing diffusive flux and a relaxed  initial condition. The paper also provides a rigorous justification for  the informal analysis of the ad hoc large deviation principle for the KPZ equation appearing in the physics literature.
\end{abstract}
\maketitle


\section{introduction}

We study an innovative connection between a very singular  optimal transport variational problem with a congestion term connecting two Dirac measures, which is related to the planing problem in mean field games or gases with negative pressure, and a large deviation principle for the KPZ equation. The main result proves a very singular limit of a problem with vanishing diffusive flux and a relaxed initial condition. The paper also provides a rigorous justification for  the informal analysis of the ad hoc large deviation principle for the KPZ equation appearing in the physics literature.
\vs

The novelty of the paper is that, using a number of pde-analysis ideas, it carries  out the limit using the Euler-Lagrange equations instead of the strictly convex energy functionals. The resulting convergence of the energy functionals provides a proof for the underlying Ventcel-Freidlin-type large deviations principle (LDP). The convergence of the Euler-Lagrange system also provides an analysis for the asymptotic behavior of the underlying fundamental solutions, which,  due to the singularities at initial and final times,  is always much harder to justify  rigorously  than the LDP. 
\vs

The ad hoc LDP formulated in the physics literature introduces the  space-time variational problem 
\be\label{takis500}
\begin{split}
I_{\ep, \eta}= \inf \Big \{ &\int_\R \rho(x,0) \dfrac{x^2}{2\eta} dx + \int_0^1\int_\R \dfrac{1}{2} \big[\rho^2 +{\rho}{\alpha}^2 \big] dx dt: \\
& \pt_t \rho  +\ep \pt^2_x \rho + \pt_x (\rho \alpha)=0 \ \text{in} \ \rt,  \ \rho(\cdot,1)= \text{\boldmath $\delta$} \  \text{and} \ \int_0^1\int_\R {\rho}\alpha^2 dx dt<\oo \Big \},
\end{split}
\ee
which, as it is claimed without a proof, must converge to 
\be\label{takis501}
\bs
\overline I=&\inf\Big\{ \ds\int_0^1 \ds\int_\R \dfrac{1}{2}\big[\rho^2 + \rho \alpha^2] dx dt:\\
& \pt_t \rho + \pt_x(\rho \alpha)=0  \ \text{in} \ \rt,  \  \rho(\cdot,0)=\rho(\cdot,1)= \text{\boldmath $ \delta$}  \ \text{and}  \ \int_0^1\int_\R {\rho} \alpha^2 dx dt<\oo  \Big\},
\end{split}
\ee
which characterizes the LDP property. In \eqref{takis500} and \eqref{takis501},  $\text{\boldmath $ \delta$}$ is the Dirac mass at $x=0$. 
\vs

Both problems are similar to the Benamou-Brenier formulation of optimal transport but for the inclusion of the congestion terms 
\[ \frac{1}{2}  \int_0^1\int_\R\rho^2,\]
which provides the coupling between individual trajectories. 
\vs

The main difference between $I_{\ep, \eta}$ and  $\overline I$ is that the former includes a small diffusion term, parametrized by $\ep$, which breaks the symmetry in time of the latter. This lack of reversibility of the underlying dynamics   requires the relaxation of the singular behavior at the initial time 
replacing the Dirac mass by a large quadratic introducing another small parameter $\eta$. 
\vs

The main theorem of the paper, which will be stated below, proves the convergence, as $\ep, \eta \to 0$,  of the minimal energies under the rather weak assumption that 
\be\label{uc0}  
\text{there exists $\kappa>1$ such that}
\underset {\ep\to 0,\eta \to 0} \lim \ep (-\log \eta)^\kappa=0. 
\ee


As already mentioned above this passage to the limit is accomplished using the convergence of the corresponding Euler-Lagrange systems which we describe next. This convergence is the first part of the main result of the paper. 
\vs

The Euler-Lagrange system for either  variational problems features, in addition to the continuity equation of the density,   a second equation satisfied by the Lagrange multiplier for the continuity equations. 
\vs 
For  $I_{\ep, \eta}$, this equation is a viscous  Hamilton-Jacobi equation, with viscosity coefficient $\ep$, while for $\overline I$ is a Hamilton-Jacobi equation. 
This is similar to the classical optimal transport theory but for the coupling. This is, of course, the main source of all the difficulties we face here.
 \vs
The Euler-Lagrange system for  \eqref{takis500} and \eqref{takis501} is respectively  
\be\label{takis1}
\bc
\pt_t u_{\ep, \eta} -{\ep} \pt^2_{x}u_{\ep, \eta} + \dfrac{1}{2}(\partial_x u_{\ep, \eta})^2 = \rho_{\ep, \eta}  \; \; \text{in} \; \; \R\times (0,1) \ \ \text{and} \ \ u_{\ep, \eta}(x,0)=x^2/2\eta,\\[2mm]
\pt_t \rho_{\ep, \eta} + \ep \pt^2_{x} \rho_{\ep, \eta}+\pt_x( \rho_{\ep, \eta}\pt_x u_{\ep, \eta})=0 \;\; \text{in} \;\; \R\times (0,1) \ \ \text{and} \ \  \rho_{\ep, \eta}(\cdot,1)= \text{\boldmath $ \delta$},
\ec
\ee
and 
\be\label{takis2}
\bc
\pt_t \overline u + \dfrac{1}{2}(\partial_x \overline u)^2 = \overline \rho \; \; \text{in} \; \; \R\times (0,1)  \ \ \text{and} \ \  \overline u(0,0)=0, \\[2mm]
\pt_t  \overline \rho+ \pt_x( \overline \rho \partial_x \overline u )=0 \;\; \text{in} \;\; \R\times (0,1) \ \ \text{and} \ \ 
 \overline \rho (\cdot,0)=\text{\boldmath $ \delta$}, \ \ \ \overline\rho(\cdot,1)= \text{\boldmath $ \delta$}.
 \ec
\ee
It can be easily seen that \eqref{takis1} and \eqref{takis2} are satisfied respectively by the unique minimizers  $(\rho_{\ep,\eta}, \pt_x u_{\ep,\eta})$ 
of 
\eqref{takis500} and $(\ol \rho, \pt_x \ol u)$ 
of \eqref{takis501}. 
\vs
We remark  here that the introduction of the diffusion in  \eqref{takis1} breaks the time-reversal symmetry $t$ to $1-t$ of \eqref{takis2} in the support of $\ol \rho$. The backwards diffusion for $\rho$ makes it impossible to have $ \rho_{\ep,\eta} (\cdot,0)=\text{\boldmath $ \delta$}$. Indeed implementing this initial condition will require having a Brownian bridge with a square integrable drift, which is not possible. This gives rise to the need to relax the potential equation with the introduction of the initial data $x^2/(2\eta)$. 
\vs

The main result of the paper is the following theorem. 

\begin{thm}\label{main}
For all $\ep, \eta \to 0$ satisfying  \eqref{uc0}, 
\vskip.05in
$(i)~u_{\ep,\eta} \to \ol u \ \text{locally uniformly in} \  \R\times (0,1],  \ \rho_{\ep,\eta} \to \ol \rho \ \ \text{in} \ \ L^2(\rt), \\[1.2mm]
 \pt_x u_{\ep,\eta} \to \pt_x \ol u \ \text{in} \ L^2_{\text{loc}}( \R\times (0,1] ) \ \text{and}  \
\sqrt{\rho_{\ep,\eta}} \pt_x u_{\ep,\eta} \to \sqrt{\ol \rho} \pt_x \ol u \ \ \text{in} 
 \ L^2( \rt ).$

\vskip.05in
(ii)~$I_{\ep,\eta} \to \ol I.$
\end{thm}

We note that obviously  the result also holds when $\rho_{\ep,\eta}(\cdot, 1)= c \text{ \boldmath $ \delta$}$. We choose the constant to be $1$ to simplify the presentation.
\vs

The system \eqref{takis2} represents a one-dimensional planning mean field game  problem in which the player wants to move a Dirac mass at $t=1$ to a Dirac mass at $t=0$. The facts that (i)~these two  measures are singular and (ii)~the Dirac mass at $t=0$ is created by the ``singular'' limit, $\eta \to 0$, in \eqref{takis1} make the passage to the limit difficult and delicate.   In addition, the convergence of the $\rho_{\ep,\eta}$'s requires special attention, since the  estimates on $\rho_{\ep,\eta}$ and $u_{\ep,\eta}$ only yield convergence in $L^2_{\text{loc}}(\rt)$ which is not enough to prove the convergence of the $I_{\ep, \eta}$'s. The global $L^2-$ convergence is established  using the mean field structure of \eqref{takis1} and \eqref{takis2}, by estimating directly  $\|\rho_{\ep,\eta}-\ol \rho\|_{L^2(\rt)}$,  which requires a detailed understanding of the limiting behavior of 
some new terms. As already mentioned we do not prove directly that the $I_{\ep,\eta}$'s converge to $\ol I$ despite the strong convexity of the functionals involved. Indeed, the singularities due the Dirac masses seem to prevent such a direct approach. Finally, we remark  here that \eqref{takis1} cannot be solved if we start with $u_{\ep, \eta}(\cdot,0)= \textbf{1}_{\{0\}}$, that is,  $u_{\ep, \eta}(0,0)=0$ and $u_{\ep, \eta}(\cdot,0)=\oo$ otherwise, 
which in turn will require $\rho_{\ep,\eta}(\cdot,0)=\text{\boldmath $ \delta$}$. 
Hence, it is necessary to introduce the penalization $x^2/2\eta$. 
\vs

We note that  Theorem~\ref{main}  holds for any  convex Hamiltonian $H$ with second derivative uniformly bounded from above and below away from $0$. Here, in view of the connection with  the LDP for the KPZ equation,  we concentrate only a quadratic one. 
\vs

%

Our motivation for the work in this paper was to provide a rigorous justification to an  ad hoc formulation, which ignores the  singular nature of the equation,  of the short, that is, fixed, time  large deviations of the  $1+1$ KPZ equation that  was derived in the physics literature by Kolokolov and Korshunov \cite{KK07,KK09}, Meerson, Katzav and Vilenkin \cite{MKV16} and Kamenev, Meerson and Sasor \cite{KMS16}.
These references used   the  weak-noise theory (WNT), which grew from the Martin, Siggia and Rose \cite{MSR}  path integral  formalism in physics,  and the Ventcel-Freidlin large deviation theory  in mathematics \cite{FW}, to study the tails of the  one-point probability distribution  of height modeled by the KPZ equation. Some of the physics arguments were discussed by Lin and Tsai \cite{LT}. 
Although the rate functionals introduced in the aforementioned works are the dual of  \eqref{takis500} to \eqref{takis501}, their Euler-Lagrange equations, which were formulated in \cite{KMS16},   are exactly the same with \eqref{takis1} and \eqref{takis2}. 
\vs
In the language of the KPZ equation, the limit $\eta \to 0$ corresponds to the narrow edge initial condition, while letting $\ep\to 0$ corresponds to a spatial zoom-out of the KPZ equation for fixed times. We note that,  in spite of several claims to the opposite, what was lacking in the  literature was a complete proof of the LDP which requires knowing the convergence of the variational rate functionals.

\vs

\vs
After rescaling, the KPZ equation  takes the form
\be\label{takis00}
\pt_t h - \partial_{x}^2 h +  \dfrac{1}{2}(\partial_x h)^2 =\sqrt{\theta} \xi \ \ \text{in} \ \ \R\times (0,1) \ \ \ \ \  h(x,0)=x^2/2\eta,
\ee
where  $\xi=\xi(x,t)$ is a Gaussian space-time white noise with $\langle \xi(x,t) \rangle =0$ and $\langle \xi(x_1,t_1)\xi (x_2,t_2) \rangle =\delta (x_1-x_2)\delta(t_1-t_2)$. 
\vs

Since the Cameron-Martin space of functions $\xi$ in $L^2$ consists of square integrable random variables, we know that, for any $\rho\in L^2(\rt)$,
\[ -\log \P(\sqrt{\theta} \xi) \approx \|\rho\|^2_{L^2(\rt})=\frac{1}{\theta} \ds\int_0^1\ds \int_\R \rho^2 dx dt.\]
In the set $\{\omega :\; \sqrt{\theta } \xi \approx  \rho\}$, \eqref{takis00} is approximated by 
\be\label{takis000}
\pt_t h - \partial_{x}^2 h + \dfrac{1}{2} (\partial_x h)^2 =\rho \ \ \text{in} \ \ \R\times (0,1) \ \ \ \ \  h(x,0)=x^2/2\eta.
\ee

%

The LDP theory suggests that the logarithm of the
probability that the solution to \eqref{takis00}  satisfies $h(0,1)\geq 1/\ep$  is given
by the supremum

\be\label{takis0000}
\bs
\hat I_{\ep,\eta}=\sup &\Big \{- \ds \int_0^1 \ds \int_\R \rho^2 dx dt: \;  \\[1.2mm]
& \pt_t u - \pt_x^2 u +\dfrac{1}{2} (\pt_x u)^2=\rho \ \text{in} \ \rt, \ u(x,0)=x^2/2\eta,  \  u(0,1)=1/\ep \Big \},                      
\end{split}
\ee
whose minimizers satisfy the system
\be\label{takis0.0}
\bc
& \pt_t \hat u^{\ep,\eta} -\pt^2_x \hat u^{\ep,\eta} +\dfrac{1}{2} (\pt_x \hat u^{\ep,\eta})^2=\hat \rho^{\ep,\eta} \ \ 
\text{in} \ \rt, \\ 
& \pt_t \hat \rho^{\ep,\eta} + \pt_x^2\hat  \rho^{\ep,\eta} +\pt_x(\hat \rho^{\ep,\eta}\pt_x \hat u^{\ep,\eta})=0 \ \ \text{in} \ \ \rt,\\ %
& \hat u^{\ep,\eta}(x,0)=x^2/2\eta,   \hat u^{\ep,\eta}(0,1)=1/\ep \ \ \text{and} \ \ \hat \rho^{\ep,\eta}(\cdot,1)=c(\ep,\eta) \text{\boldmath $ \delta$}.
\ec
\ee
\vs
At first reading \eqref{takis0.0} appears to be illposed  due to the condition on $\hat u^{\ep,\eta}(0,1)$. But this condition is reflected in the constant in front of $ \text{\boldmath $ \delta$}$.
\vs

We remark that standard  duality/variational arguments  show that, up to appropriate rescalings,   \eqref{takis0000} and \eqref{takis500} are equivalent. Since the details are tedious we leave the details to the interested reader. 

\vs
Upon the rescaling $\ds \hat u_{\ep,\eta}(x,t)=\ep^{-1} \hat u^{\ep,\eta}(\ep^{1/2} x,t)$ and $\ds \hat \rho_{\ep,\eta}(x,t)=\frac{1}{\ep} \hat \rho^{\ep,\eta}(\ep^{1/2}x/,t)$, \eqref{takis0.0} becomes
\be\label{takis0.00}
\bc 
& \pt_t \hat u_{\ep,\eta} - \ep\pt^2_x \hat u_{\ep,\eta} +\dfrac{1}{2} (\pt_x \hat u_{\ep,\eta})^2=\hat \rho_{\ep,\eta} \ \ 
\text{in} \ \rt,\\ 
& \pt_t \hat \rho_{\ep,\eta} + \ep  \pt^2_x \hat \rho_{\ep,\eta} +\pt_x(\hat\rho_{\ep,\eta} \pt_x \hat u_{\ep,\eta})=0 \ \ \text{in} \ \ \rt, \\
& \hat u_{\ep,\eta}(x,0)=x^2/2\eta,  \ \ \hat u_{\ep,\eta}(0,1)=1 \ \ \text{and} \ \  \hat\rho_{\ep,\eta}(\cdot,1)=\ep^{3/2}c(\ep,\eta) \text{\boldmath $ \delta$}.
\ec
\ee

We note that the only differences between \eqref{takis0.00} and \eqref{takis1} are  the constraint  $\hat u_{\ep,\eta}(0,1)=1$ and the factor $\ep^{3/2}c(\ep,\eta)$ in the terminal condition for $\hat \rho_{\ep,\eta}$.
\vs

After proving  that $$\underset{\ep,\eta,\ep (-\log \eta)^\kappa\to } \lim \ep^{3/2}c(\ep,\eta)=c_0,$$  we  use Theorem~\ref{main} to pass in the limit 
$\ep,\eta,\ep (-\log \eta)^\kappa \to 0$ to obtain the unique solution to \eqref{takis2} with $\hat u(0,1)=1$ and $\hat \rho(\cdot,1)=c_0 \text{ \boldmath $ \delta$}$. We remark again that this constant is needed in order to enforce the constraint on $\hat u(0,1)$. 
\vs
This is the topic of the second theorem of the paper. To state it, we introduce  the minimization problem
\be\label{ph145}
{\hat I}= \sup \Big \{- \ds \int_0^1 \ds \int_\R \rho^2 dx dt: \; 
\pt_t u +\dfrac{1}{2} (\pt_x u)^2=\rho \ \  \text{in} \ \rt \ \  \text{and}  \  \ u(0,1)=1 \Big \},
\ee
which, as remarked above,  is known to be the dual  to $\ol I$.
\vs
\begin{thm}\label{main2}
Assume \eqref{uc0}. Then
\smallskip

(i)~$\underset{\ep,\eta\to 0} \lim \ep^{3/2}c(\ep,\eta) = c_0$ and $\hat u_{\ep,\eta}\to \hat u$ locally uniformly in $\R\times (0,1]$ and $\hat \rho_{\ep,\eta} \to \hat  \rho$ in $L^2(\R\times (0,1))$,  where $(\hat u, \hat  \rho)$ is the unique solution to \eqref{takis2} such that $\hat u(0,1)=1$ and $\hat \rho(\cdot,1)= c_0 \text{ \boldmath $ \delta$}$, 
\smallskip

(ii)~$\underset{\ep,\eta \to 0} \lim \hat I_{\ep, \eta}=\hat I.$

\end{thm}

We remark that, contrary to other references, here we do not rely  at all on any tools from the theory of large deviations to identify the limit of the $\hat I_{\ep,\eta}$. This is important since the classical LPD theory applies to problems with ``regular'' initial conditions. 

\vs
Here, we are interested in the asymptotics of transition probabilities, which requires different arguments. Our approach is to  provide a direct proof using the convergence  of the Euler-Lagrange systems.
 

\vs

The paper is organized as follows. In section~2 we discuss two  first-order planning problems. In section~3 we 
obtain uniform in $\ep$ and $\eta$ estimates for $u_{\ep,\eta}$ and $\rho_{\ep,\eta}$ in $\R\times (0,1)$.
Section~4 is about the proof of Theorem~\ref{main}. Finally,   in section~5 we provide a short discussion about the KPZ equation and the associated LDP, 
and prove Theorem~\ref{main2}.

\subsection*{Notation}  Throughout the paper given $a,b>0$ we denote by $a\lesssim_c b$ that fact, for some $C>0$  depending on $c$ we have $a\leq C b$, and by $a\lesssim b$ the fact that there exists $C>0$ which depends only the data such that $a\leq b$.

\subsection*{Acknowlegments} The authors would like to thank the anonymous referee for having read  the paper very carefully and for the many suggestions which helped to improve  the manuscript.
\vs
The second author  was partially  supported by the National Science Foundation grants DMS-1900599,  DMS-2153822 and DMS-2452972, the Office for Naval Research grant N000141712095 and the Air Force Office for Scientific Research grant FA9550-18-1-0494.



\section{the planing problem}

We begin  recalling  the almost explicit  smooth solutions to  \eqref{takis2} which were introduced in the physics literature (see \cite{KMS16, MKV16}) both for the sharp and relaxed conditions at $t=0$ and $t=1$. The solution to  the first equation in  \eqref{takis2} is  then extended to be viscosity solution in the whole domain $\rt$, a fact which is used to derive more properties in Proposition~\ref{prop1.0} and Proposition~\ref{prop11}. 
We then show that $(\ol u,\ol \rho)$ is the unique minimizer to \eqref{takis501} and  the unique solution to \eqref{takis2}.
\vs
\subsection*{The physics solution}
Planing mean field games problems were introduced by the first author in \cite{college de france}. In spite of some subsequent work  
there are no results available, when the initial and terminal data are singular measures. In particular, nothing was known about the 
system  \eqref{takis2} as well as its ``regularized''  version 
\be\label{takis200.0}
\bc
\pt_t \overline u_\eta + \dfrac{1}{2}(\partial_x \overline u_\eta)^2 = \overline \rho_\eta \; \; \text{in} \; \; \R\times (0,1) \ \ \text{and} \ \  \overline u_\eta (x,0)= x^2/2\eta, \\[2mm]
\pt_t  \overline \rho_\eta + \pt_x(  \overline \rho_\eta \partial_x \overline u_\eta  )=0 \;\; \text{in} \;\; \R\times (0,1)  
 \; \; \text{and} \; \;  \ol \rho_\eta (\cdot,1)= \text{\boldmath $\delta$}.
 \ec
\ee


%

\vs 

The  density component of the solutions proposed in
\cite{KMS16, MKV16} for \eqref{takis2} and \eqref{takis200.0} are  respectively  
\be\label{takis11}
\ol \rho(x,t)= r(t) \; \Big(1- \Big(\dfrac{x}{l(t)}\Big)^2 \Big)_+ \ \text{and} \  \ \ol \rho_\eta(x,t)= r_\eta(t) \; \Big(1- \Big(\dfrac{x}{l_\eta(t)}\Big)^2 \Big)_+.
\ee
For the boundary conditions to hold, $r$, $r_\eta$, $l$ and $l_\eta$ must satisfy the properties 
\be\label{oxford11}
\lim_{t\to 0} r(t)=\lim_{t\to 1} r(t)=\lim_{t\to 1} r_\eta (t)=\infty 
\ee
and, for all $t\in (0,1)$,
\be\label{oxford12}
r (t)l(t)=3/4  \ \ \text{and} \ \ r_\eta(t)l_\eta(t)=3/4.   
\ee
The other components of the solutions are defined initially 
in the domains $\mathcal O=\{(x,t): |x|\leq l(t)\}$ and $ \mathcal O_\eta=\{(x,t): |x|\leq l_\eta(t)\}$ respectively and are given by 
\be\label{takis12}
\ol u(x,t)=k(t) + \dfrac{1}{2} a(t) x^2 \ \ \text{and} \ \ \ol u_\eta(x,t)=k_\eta(t) + \dfrac{1}{2} a_\eta(t) x^2.
\ee
Inserting $(\ol u, \ol\rho)$ and  $(\ol u_\eta, \ol\rho_\eta)$  in \eqref{takis2} and \eqref{takis200.0} we find that 
that $k$, $\rho$ and $a$ and $k_\eta$, $\rho_\eta$,  and $a_\eta$ must satisfy  the odes
\be\label{takis13.1}
\dot k=r, \ \ \dot r+r a=0 \ \ \text{and} \ \  \dot a + a^2 + \dfrac{32}{9}r^3=0
\ee
and 
\be\label{takis13}
\dot k_\eta=r_\eta , \ \ \dot r_\eta + r_\eta a_\eta=0 \ \ \text{and} \ \  \dot a_\eta + a_\eta^2 + \dfrac{32}{9}r_\eta^3=0.
\ee

As far $a$ and $a_\eta$ are  concerned,  it is assumed that, respectively,
\be\label{oxford13}
\lim_{t\to \pm \oo} a(t)=\pm \oo \ \ \text{and} \ \  \lim_{t\to 1} a_\eta(t)=-\infty \ \ \text{and}   \ a_\eta(0)=1/\eta.  
\ee
Since the Hamilton-Jacobi equation of \eqref{takis2}  is invariant under translations by constants,  we normalize the solution to \eqref{takis2} to satisfy  $\ol u(0,0)=0$. Moreover, we want to have   that, as $\eta\to 0$,  $\ol u_\eta \to \ol u$. So, in view of the normalization above, we assume 
that 
\be\label{oxford19}
k(0)=k_\eta(0)=0.
\ee
It follows from \eqref{oxford13} that there must exist $t^\ast, t_\eta^\ast \in (0,1)$ such that 
\be\label{oxford14}
a(t^\ast)=a_\eta(t_\eta^\ast)=0.
\ee
Since \eqref{takis2} is symmetric around $t=1/2$, it follows that  $t^\ast=1/2$, while identifying $t_\eta^\ast$ is part of solving the odes for $r_\eta$ and $a_\eta$. 
\vs

Using \eqref{oxford14} and the fact that the two last odes in \eqref{takis13.1}  and \eqref{takis13} combine to 
\[\ds \frac{d}{dr}\Big(\frac{a}{r}\Big)=\frac{32}{9} \frac{r}{a} \ \ \text{and} \ \ \frac{d}{dr}\Big(\frac{a_\eta}{r_\eta}\Big)=\frac{32}{9} \frac{r_\eta}{a_\eta},\] 
we obtain, for  $r^\ast=r(1/2)$ and $r_\eta^\ast=r_\eta(t_\eta^\ast)$, that  
\be\label{oxford15}
\bc
a=\frac{8}{3} r\sqrt{r-r^\ast} \ \ \text{and} \ \  \dot r=-\frac{8}{3} r^2\sqrt{r-r^\ast} \  \text{in} \  (0,1/2),\\[1.2mm]
a=-\frac{8}{3} r\sqrt{r-r^\ast} \ \ \text{and} \  \ \dot r=\frac{8}{3} r^2\sqrt{r-r^\ast} \  \text{in} \ (1/2,1),
\ec
\ee
and
\be\label{oxford16}
\bc
a_\eta=\frac{8}{3} r\sqrt{r-r_\eta^\ast} \ \ \text{and} \ \ \dot r_\eta=-\frac{8}{3} r^2\sqrt{r-r_\eta^\ast} \  \text{in} \  (0, t_\eta^\ast),\\[1.2mm]
a_\eta=-\frac{8}{3} r\sqrt{r-r_\eta^\ast} \ \ \text{and} \ \ \dot r_\eta=\frac{8}{3} r^2\sqrt{r-r_\eta^\ast} \  \text{in} \ (1/2, t_\eta^\ast).
\ec
\ee
\vs
\vs 
The difference between $(a,r)$ and $(a_\eta,r_\eta)$ is the boundary condition at $t=0$, since $\lim_{t\to 0} r(t)=\lim_{t\to 0} a(t)=0$ while $a_\eta(0)=1/\eta$ and $1/\eta=-(8/3)r_\eta(0)\sqrt{r_\eta(0) - r_\eta^\ast}$.
\vs
Writing the ode for $r$ as 
\[\dot r=\pm r^{3/2}\sqrt{1-(r^\ast/r)},\]
we see that  
\[\dot r \approx -  r^{3/2} \ \text{as} \ t\to 0 \ \ \text{and} \ \  \dot r \approx r^{3/2} \ \text{as} \ t\to 1 .\]
This observation together with the equation for $a$ in \eqref{oxford15} as well a similar reasoning for $r_\eta$ and $a_\eta$  imply that
\be\label{takis210}
\begin{split}
\ds & r(t)\approx \dfrac{1}{(4t)^{2/3}}, \  a(t)\approx \dfrac{2}{3t}, \ 
r_\eta(t)\approx \dfrac{1}{(4(t+\eta))^{2/3}}, \  a_\eta(t)\approx \dfrac{2}{3(t+\eta)} \ \text{as} \ t\to 0,\\
& \text{and}\\
& r(t)\approx \dfrac{1}{(4(1-t))^{2/3}}, \ a(t)\approx -\dfrac{2}{3}\dfrac{1}{1-t}, \  
r_\eta(t) \approx  \dfrac{1}{(4(1-t))^{2/3}}, \ a_\eta(t) \approx -\dfrac{2}{3(1-t)} \ \text{as} \ t\to 1.
\end{split}
\ee
\vs

We remark that, if $\ol \rho\equiv 0$, then, near $t=0$, $a(t)=1/(2t)$. The difference in the coefficient in front of $1/t$ is a result of the presence of $\ol \rho$ and the singularity of $\ol \rho$ near $t=0$.  A similar argument applies near $t=1$.
\vs

In view of the equations satisfied by $k$ and $k_\eta$ we also have  
\be\label{211.11}
k(1) <\oo \ \ \text{and} \ \ k_\eta(1)<\oo.
\ee

Next, we extend the solutions to the whole domain $\rt$. This is an easy task  on the level of  the densities, since we let
\be\label{united20}
\ol \rho=0 \ \text{in} \   \R\times (0,1] \setminus {\ol {\mathcal O}} \ \ \text{and} \ \ \rho_\eta=0 \ \text{in} \ \R\times (0,1] \setminus {\ol {\mathcal O_\eta}}.
\ee
\vs


The exact smooth solution obtained in the physics literature in  ${\ol {\mathcal O}}$ is symmetric in time across $1/2$ leading to the possibility to extend $\ol u$ in the whole space $\rt$ preserving the symmetry. However, this is not the desired extension. Indeed the solution obtained in the $\ep,\eta\to 0$ limit of (viscosity) solutions to parabolic pde must be a viscosity solution and, thus, it can not be globally symmetric in time. 
\vs

%
%
%

Hence, we extend  $\ol u$ and $\ol u_\eta$ outside ${\ol {\mathcal O}}$  and ${\ol {\mathcal O_\eta}}$ using the method of characteristics, and we present some the details below.  
%
\vs
Given $(x_0,t_0)\in \R\times (0,1] \setminus {\ol {\mathcal O}}$, we take the straight line in the direction of the axis $t=0$, which is  tangent to $\pt \mathcal O$ at some point $(l(\ol t_0),\ol t_0)$. 
For any point $(x,t)$ of the line segment connecting $(x_0,t_0)$ and  $(l(\ol t_0),\ol t_0)$, $\ol u(x,t)$ is defined 
$\ol u$  by 
\be\label{takis14.0}
\ol u(x,t)=\ol u(l(t_0),t_0) + \dfrac{1}{2} \dot l(t_0)^2 (t-t_0)=\ol u(l(t_0),t_0) + \dfrac{(x-l(t_0))^2}{2(t-t_0)}.
\ee
It is a straightforward exercise to check that the $\ol u$, which was constructed earlier in $\ol{\mathcal O}$ and by \eqref{takis14.0} in $\R\times (0,1] \setminus {\ol {\mathcal O}}$ is a continuous viscosity solution to the Hamilton-Jacobi equation in  \eqref{takis2}. For the convenience of the reader,  we sketch the argument omitting the calculus-type details. 
\vs

In $\mathcal O$, the claim is obvious since $\ol u$ is smooth. Moreover, for any $(x,t) \in \R\times (0,1] \setminus {\ol {\mathcal O}}$,  it is immediate from \eqref{takis14.0} that $\pt_t\ol u$ and $\pt_x\ol u$ are  constant along the segment  connecting $(x,t)$ and the tangent point $(l(t_0), t_0)$ and equal to $\pt_t\ol u(l(t_0), t_0)$ and $\pt_x\ol u(l(t_0), t_0)$. So the equation is satisfied  in the viscosity sense. Finally, it is immediate from \eqref{takis14.0} and the discussion above that $\ol u$ is $C^1$ across $\pt \mathcal O$ and, hence, a viscosity solution by construction.
\vs

For $\ol u_\eta$, we recall that $\mathcal O_\eta$  intersects the $t=0$ axis at $\pm r_\eta(0)$. The extension described above for $\ol u$ also works for $\ol u_\eta$ for any point of 
$\R\times (0,1]\setminus {\ol {\mathcal O_\eta }}$ above the tangent to ${\ol {\mathcal O_\eta }}$ at the points $(\pm r_\eta (0), 0)$. For points $(x,t)$ below these two lines, we define $\ol u_\eta(x,t)$ as $\ol u_\eta(x,t)=x^2/2(t+\eta)$.
We leave it to the reader to check that $\ol u_\eta$ is continuous across the tangent lines and also a viscosity solution.

\vs
We conclude the construction noting that, in view of \eqref{oxford19} and \eqref{takis210}, we have 
\be\label{oxford20}
\ol u(0,0)=0  \ \ \text{and} \ \  \ol u_\eta(x, 0)=\frac{x^2}{2\eta}.
\ee
Finally, for future use we note that, in view of  as  \eqref{oxford19}, \eqref{oxford15}, \eqref{takis210} and the extensions described above,  we also find that
\be\label{united21}
\bs
& \ol u \ \text{is locally bounded and continuous in} \ \R\times(0,1] \ \ \text{and}\\
& \ol u_\eta  \ \text{is locally bounded and continuous in} \  R\times[0,1].
\end{split}
\ee
For the rest of the paper we refer to  the  $(\ol u,\ol \rho)$ (resp. $(\ol u_\eta, \ol \rho_\eta)$) constructed above as the viscosity (resp. approximate)  solution.
\vs

The following proposition is immediate from all the previous discussion, so we state it without proof.

\begin{prop}\label{prop.oxford}
The pairs $(\ol u, \ol \rho)$ and $(\ol u_\eta, \ol \rho_\eta)$ are solutions to  \eqref{takis2}  and \eqref{takis200.0} in $\R\times (0,1)$  respectively, where the first equation of each system  is understood in the viscosity sense and the second in the sense of distributions. Moreover, as $\eta\to 0$, $\ol u_\eta \to \ol u$ locally uniformly in $\R\times (0,1)$ and $\ol \rho_\eta \to \ol \rho$ in $L^2(\R\times (0,1))$.
\end{prop}

\subsection*{Properties of the viscosity solutions to the exact and  and approximate problems} 
In the next two propositions we record the properties of $\ol u$ and $\ol \rho$ and $\ol u_\eta$ and $\ol \rho_\eta$ that we  will use later in the paper.
\vs

\begin{prop}\label{prop1.0}
(i)~For any $\theta \in (0,1)$ and $(x,t)\in \R\times (\theta,1)$,
\be\label{takis20.1}
\ol u(x,t)=\inf \big\{\int_{\theta}^t \big[\dfrac{(\dot x(s))^2}{2} + \ol \rho(x(s),s)\big] ds + \ol u(x(\theta),\theta) : \; x\in H^1((0,1)) \ \text{and} \  x(t)=x \big\}.
\ee

(ii)~For any $(x,t)\in \R\times (0,1)$,
\be\label{takis21.0}
\ol u(x,t)=\inf \big\{\int_0^t \big[\dfrac{(\dot x(s))^2}{2} + \ol \rho(x(s),s)\big] ds : \; x\in H^1((0,1)),   \ x(t)=x \ \text{and} \ x(0)=0 \big\}.
\ee

(iii)~For any $(x,t) \in \R\times (0,1)$,
\be\label{takis22.0}
\ol u(x,t)\geq \dfrac{x^2}{2t}.
\ee

(iv)~For all sufficiently small $t>0$ so that \eqref{takis210} holds and for all $\alpha>0$,
\be\label{takis23.0}
\sup_{|x|\geq \alpha} \big[\ol u(x,t)-\dfrac{x^2}{2t}\big] \lesssim \frac{t}{\alpha}. 
\ee
\end{prop}
\begin{proof}
The dynamic programming principle \eqref{takis20.1}  is  an immediate  consequence of the fact that $\ol u$ is a viscosity 
solution to
\[\pt_t \ol u +\dfrac{1}{2} (\pt_x \ol u)^2= \ol \rho \ \text{in} \ \rt. \]

\vs

By construction,  we also know that 
\[\ol u(\cdot,t)\geq 0 \ \text{ for all \  $t\in (0,1]$}.\]


\vs

Next,  observe that, if $x\in \R\setminus\{0\}$ and $t$ is sufficiently small, $(x,t) \notin \ol {\mathcal{O}}$ and $\ol u$ is given by \eqref{takis14.0} for $t_0<t$. Then, as $t\to 0$, it is immediate that $t_0\to 0$ and $l(t_0)\to 0$, and, hence,  
\be\label{takis24}
 \underset{t\to 0} \lim \; \ol u(x,t)=\oo, 
\ee
since 
\[
\ol u(x,t)=\ol u(l(t_0),t_0) +\dfrac{(x-l(t_0))^2}{2(t-t_0)} \geq \dfrac{(x-l(t_0))^2}{2(t-t_0)} \ \ \text{and} \ \ \underset{t\to 0} \lim \dfrac{(x-l(t_0))^2}{2(t-t_0)}\geq \underset{t\to 0} \lim \dfrac{(x-l(t_0))^2}{2t}=\oo.
\]
%
\vs 
Using this information in \eqref{takis20.1}, we find that, if $x_\theta$ is an approximate minimizer, then $\underset{\theta \to 0}\lim x_\theta(\theta)=0$, which implies that $\ol u$ must actually satisfy \eqref{takis21.0}.
\vs

Then  \eqref{takis21.0} and the fact that $\ol \rho \geq 0$ imply that 
\[\ol u(x,t)\geq \inf \big\{\int_0^t \frac{1}{2} (\dot x(s))^2  ds : \; x\in H^1((0,1)),   \ x(t)=x \ \text{and} \ x(0)=0 \big\},\]
and the lower bound in \eqref{takis22.0} follows now from Jensen's inequality. 
\vs

For each $(x,t) \in \big(\R \setminus \{0\}\big) \times (0,1)$, the line with slope $x/t$ intersects 
$\pt \mathcal O$ at some $\hat t$, that is, $(|x|/t) \hat t=l(\hat t)$.   Then, it follows from \eqref{takis22.0} that
\vs
%
\be\label{takis171.1}
\ol u(x,t) \leq \dfrac{x^2}{2t} + \int_0^{\hat t}\ol \rho (\frac{x}{t} s, s) ds.
\ee
Using the assumption on $t$, it follows that 
\[ \int_0^{\hat t}\ol \rho (\frac{x}{t} s, s) ds \approx \frac{t}{|x|}, \]
and, thus,  the claimed estimate \eqref{takis23.0}.

\end{proof}
\vs

%
%
%
The result about the properties of $\ol u_\eta$ and $\ol \rho_\eta$ is stated next. Since the proofs of the first three claims are identical with one of the previous proposition and the last assertion follows from the construction of $\ol u_\eta$, we state the next proposition without a proof. 

\begin{prop}\label{prop11}
(i)~For any $\theta>0$ and $(x,t)\in \R\times (\theta,1)$,
\be\label{takis20}
\ol u_\eta (x,t)=\inf \big\{\int_{\theta}^t \big[\dfrac{(\dot x(s))^2}{2} + \ol \rho_\eta(x(s),s)\big] ds + \ol u_\eta (x(\theta),\theta) \; : x\in H^1((0,1)), \ \ x(t)=x \big\}.
\ee

(ii)~For any $(x,t)\in \R\times (0,1)$,
\be\label{takis21}
\ol u_\eta (x,t)=\inf \big\{\int_0^t \big[\dfrac{(\dot x(s))^2}{2} + \ol \rho_\eta(x(s),s)\big] ds + \dfrac{x(0)^2}{2\eta} \; : x\in H^1((0,1)), \ x(t)=x,   \ x(0)=0 \big\}.
\ee

(iii)~For any $(x,t) \in \R\times (0,1)$,
\be\label{takis22}
\ol u_\eta (x,t)\geq \dfrac{x^2}{2(t+\eta)}.
\ee

(iv)~Locally uniformly in $x\in \R$, 
\be\label{takis23}
\underset{t\to 0} \lim \big[\ol u_\eta (x,t)-\dfrac{x^2}{2(t+\eta)}\big]=0.
\ee
\end{prop}
\vs

We state and prove now some important  properties about the behavior near $t=0$ and $t=1$ of  solutions to the continuity  equation.

\begin{prop}\label{prop31}
Assume that  \ $\pt_t \rho +\pt_x(\rho \alpha)=0$ in $\rt$, $\rho(\cdot,0)=\rho(\cdot,1)=\text{\boldmath $\delta$}$

and $\ds\int_0^1\int_\R \rho \alpha^2 dxds<\oo$. Then,
\be\label{phl31}
\underset{t \to 0} \lim \ds \int_\R \rho(x,t) \dfrac{x^2}{t} dx =0,\ee
and
\be\label{phl31.1}
\underset{t \to 0} \lim \ds \int_\R \rho(x,t) \ol u(x,t)dx=0 
\ee
Moreover, for each $R>0$, 
\be\label{phl32}
\underset{t\to 1} \lim \ds \int_{\{|x|\geq R\}} x^2\rho(x,t) dx =0.
\ee
%
\end{prop}
\begin{proof}
Multiplying the continuity equation satisfied by $\rho$ by   $\big(|x|\wedge R\big)^2$, with $R>0$, 
and integrating over $\R$ gives

\[\dfrac{d}{dt} \ds \int_\R  \big(|x|\wedge R\big)^2 \rho dx  = -2 \ds \int_{-R}^R( |x|\wedge R )  (\rho \alpha)(x,t)  dx \leq  2 \Big( \ds \int_\R \big(|x|\wedge R\big)^2 \rho dx\Big)^{1/2} \Big( \ds \int_\R  \rho \alpha^2 dx\Big)^{1/2},
\]
\vs
and thus
\[\dfrac{d}{dt} \Big( \ds \int_\R \big(|x|\wedge R\big)^2 \rho dx\Big)^{1/2} \leq \Big( \ds \int_\R  \rho \alpha^2 dx\Big)^{1/2}.\]

\vs
It then follows, using   $\ds\int \big(|x|\wedge R\big)^2 \rho(x,0) dx=0$, that 

\[\Big( \ds \int_\R \big(|x|\wedge R\big)^2  \rho (x,t) dx\Big)^{1/2} \leq t^{1/2} \Big( \ds \int_0^t \ds\int_\R  \rho \alpha^2 dx\Big)^{1/2}, \]
and, after letting $R\to \infty$, 

\[\Big(  \int_\R x^2  \rho (x,t) dx\Big)^{1/2} \leq t^{1/2} \Big(  \int_0^t \ds\int_\R  \rho \alpha^2 dx dt\Big)^{1/2}.\]

 
Squaring the last inequality we get

\[ \ds \int_\R \dfrac{x^2}{t}  \rho (x,t) dx \leq \ds \int_0^t \ds\int_\R  \rho \alpha^2 dx,\]

and,  in view of the $L^2(\rt)-$ integrability of $\sqrt{\rho} \alpha$,  \eqref{phl31}   follows. 
\vs
For   \eqref{phl31.1} we combine \eqref{takis12}, \eqref{oxford19}, \eqref{takis210}, \eqref{takis23.0}  and, finally,  \eqref{phl31} as well as the fact that, for all $t\in (0,1)$,  $\int_\R \rho(x,t) dx=1$. 
%

\vs
Indeed,
\be\label{phl31.15}
\int_\R \rho(x,t)\ol u(x,t) dx= \int_{\{|x|\leq l(t)\}} \rho(x,t)\ol u(x,t) dx + \int_{\{|x|\geq l(t)\}} \rho(x,t)\ol u(x,t) dx.
\ee
\vs
For the second integral we find, using \eqref{takis23.0},  
 that, for some $\ds g(t)\approx \frac{t}{l(t)}$, 
\[ \int_{\{|x|\geq l(t)\}} \rho(x,t)\ol u(x,t) dx \leq  \int_{\{|x|\geq l(t)\}} \frac{x^2}{t} \rho(x,t) dx + g(t)\]

Moreover,  in view of \eqref{takis210}, we have
\[\frac{t}{l(t)}\approx t^{1/3}.\]

Hence,
\be\label{phl31.11}
 \int_{\{|x|\geq l(t)\}} \rho(x,t)\ol u(x,t) dx \approx  \int_{\{|x|\geq l(t)\}} \frac{x^2}{t} \rho(x,t) dx + t^{1/3}.
\ee

\vs

For the first integral in the right hand side of \eqref{phl31.15}, we recall that the explicit formula for $\ol u$ in  $\{|x|\leq l(t)\}$ in  \eqref{takis12}   implies  that 

\[\int_{\{|x|\leq l(t)\}} \rho(x,t)\ol u(x,t) dx \leq \big[k(t)  + \frac{1}{2} a(t) l^2(t)].\] 

Now, \eqref{oxford19} and \eqref{takis210} give
 \[\lim_{t\to 0} k(t)=0 \ \ \text{and} \ \  a(t)l^2(t)\approx t^{1/3}.\]
 
 Combining all the above yields, for some $f(t)\approx t^{1/3}$, that 
 \[\int_\R \rho(x,t)\ol u(x,t) dx \leq \int_\R \rho(x,t)\frac{x^2}{t} \rho(x,t) dx + k(t) + f(t),\]
 and the claim follows for \eqref{phl31}.
\vs

\vs
To prove  \eqref{phl32},  we multiply  the continuity equation by $(|x|-R)_+^2$ (actually, as in the argument above,  we first need to multiply by a truncated version of $(|x|-R)_+^2$, like $\big((|x|-R)_+ \wedge M\big)^2$, and then let $M\to \oo$, but, since we already did this once, for the shake of simplicity we ignore it for this proof) and integrate by parts to find

\[\bs
\dfrac{d}{dt} \int_\R & \rho (x,t) (|x|-R)_+^2dx = -2\ds \int_{|x|\geq R} (|x|-R)_+ \rho(x,t)\alpha  (x,t)  dx\\[1.5mm]
& \geq  - \Big(\ds \int_\R \rho(x,t)(|x|-R)_+^2dx\Big)^{1/2} \Big(\ds \int_\R \rho(x,t) \alpha^2(x,t) dx\Big)^{1/2},
\end{split}
\]
and, hence, 
\[ \dfrac{d}{dt} \Big(\int_\R  \rho (x,t) (|x|-R)_+^2dx\Big)^{1/2} \geq -\Big(\ds \int_\R \rho(x,t) \alpha^2(x,t) dx\Big)^{1/2}. \]
\vs
Integrating this last inequality from $1-h$ to $1$ and using that, in view of the fact that $\rho(\cdot,1)=\text{\boldmath $ \delta$}$, 
$\ds\int_\R \rho(x,1)(|x|-R)_+^2 dx=0$, we get 
\[ \int_\R  \rho (x,1-h) [(|x|-R)_+ ]^2dx \leq h \ds \int_0^1\ds \int_\R  \rho(x,t) \alpha^2(x,t) dx dt.\]
Using that  $ \ds \int_0^1\ds \int_\R  \rho(x,t) \alpha^2(x,t) dx<\oo$, we get the claim. 
\vs

\end{proof}

We present next a technical proposition with properties of the viscosity solutions to \eqref{takis2} and \eqref{takis200.0}, which will be used in the following proofs to deal with integrals of $\ol u$ and $\ol u_\eta$ against densities near $t=0$ and $t=1$.
\begin{prop}\label{phl25}
Let $(\ol u,\ol \rho)$ and $(\ol u_\eta, \ol \rho_\eta)$ be respectively the viscosity solutions to \eqref{takis2} and \eqref{takis200.0}. Then, for any $(\rho, \alpha)$ as in Proposition~\ref{prop31}, we have
\be\label{phl1}
\underset{h\to 0} \lim \int_\R \ol u(x,1-h) (\rho -\ol \rho)(x,1-h) dx= \underset{h\to 0} \lim \int_\R \ol u_\eta(x,1-h) (\rho_\eta -\ol \rho_\eta)(x,1-h) dx=0,
\ee
and
\be\label{phl2}
\underset{h\to 0} \lim \int_\R \ol u(x,h) (\rho -\ol \rho)(x,h) dx=0.
\ee
\end{prop}
\begin{proof}
Since the arguments leading to \eqref{phl1} for $\ol u$ and $\ol u_\eta$ are similar, we only present the proof for the former.
\vs 

Recall that $\ol u$ is  continuous at $t=1$ and, in addition,  grows quadratically in space in $[1-h,1]$ for every $h\in (0,1)$.
\vs
Then, for every $\kappa>0$, there exists $R_\kappa>0$ and $h_\kappa \in (0,1)$ such that 
\be\label{phl26}
\underset{|x|\leq R_\kappa,  h\in[0,h_\kappa]} \sup |\ol u (x,1-h)-\ol u(0,1)|\leq \kappa, \ee
and
\be\label{phl27}
 |\ol u(x,1-h)|  \lesssim_\kappa |x|^2 \  \text{for all} \  |x|\geq \kappa \ \text{and} \ h\in [0,h_\kappa]. \ee
Then
\be\label{phl28} 
\begin{split}
 \int_\R \ol u(x,1-h) & (\rho-\ol \rho)(x,1-h) dx= \int_\R \ol u(0,1) (\rho-\ol \rho)(x,1-h) dx \\
&+ \int_{\{|x|\leq R_\kappa\}}( \ol u(x,1-h) -\ol u(0,1)) (\rho-\ol \rho)(x,1-h) dx\\ 
&+ \int_{\{|x| \geq R_\kappa\}} \ol u(x,1-h) (\rho-\ol \rho)(x,1-h) dx.
\end{split}
\ee
Recall that, for every $t\in (0,1)$,  
\be\label{phl29}
\int_\R \rho (x,t) dx=\int_\R \ol \rho (x,t) dx=1.
\ee
Then 
\[ \int_\R \ol u(0,1) (\rho-\ol \rho)(x,1-h) dx =0.\]
For the second term in the right hand side of \eqref{phl28} we use \eqref{phl26} and \eqref{phl29} to find
\[ |\int_{\{|x|\leq R_\kappa\}} ( \ol u(x,1-h) -\ol u(0,1)) \rho (x,1-h) dx | \leq \kappa\]
and
\[|\int_{\{|x|\leq R_\kappa\}} ( \ol u(x,1-h) -\ol u(0,1)) \ol \rho (x,1-h) dx | \leq \kappa.\]
Finally, in view of \eqref{phl28}, to conclude we need to show that,  for each $R>0$, as $t\to 1$,
\be\label{phl30}
\lim_{t \to 1} \int_{\{|x|\geq R\}} x^2\rho (x,t)=\lim_{t \to 1} \int_{\{|x|\geq R\}} x^2\ol \rho (x,t)=0.
\ee
The first of the two equalities follows from the fact that $\bar \rho(\cdot,1)=\text{\boldmath $\delta$}$. The second also follows from   \eqref{phl32} once we know that 
$\ds \int_0^1\int_\R \ol \rho(\ol \pt_x \ol u)^2dxdt<\oo$ or more directly using that $\ol \rho$ has compact support shrinking to
 $0$ as $t\to 1$.
\vs

\end{proof}


\subsection*{The variational problems} We show here that  the planning problem \eqref{takis2} and the approximate planning system \eqref{takis200.0} are  associated respectively with the strictly convex variational problems 
\be\label{takis15}
\ol I = \underset{ (\rho, \beta) \in\mathcal A}\min \Big\{\int_0^1\int_\R \dfrac{1}{2}\big[\rho^2 + \dfrac{\beta^2}{\rho}\big] dx dt \Big\} 
=\underset{ (\rho, \alpha) \in\mathcal B}\min \Big\{\int_0^1\int_\R \dfrac{1}{2}\big[\rho^2 + {\rho}\alpha^2 \big] dx dt \Big\}\ \ 
\ee
and
\be\label{takis15.111}
\bs
\ol I_\eta&=\underset{(\rho, \beta) \in  \mathcal A_\eta} \min \Big\{  
\int_0^1\int_\R \dfrac{1}{2} \big[\rho^2 + \dfrac{\beta^2}{ \rho}\big] dx dt +\int_\R \rho(x,0) \dfrac{x^2}{2\eta}  \Big\}\\
& =\underset{(\rho, \alpha) \in  \mathcal B_\eta} \min \Big\{  
\int_0^1\int_\R \dfrac{1}{2} \big[\rho^2 + { \rho}\alpha^2 \big] dx dt +\int_\R \rho(x,0) \dfrac{x^2}{2\eta}  \Big\},
\end{split}
\ee

where 
\be\label{takis241}
\begin{split}
\mathcal A=\Big\{&  (\rho,\beta):\;  \rho\in  C([0,1]; \mathcal P_2) \cap L^2(\rt):  \\
& \pt_t \rho + \pt_x \beta=0 \ \text{and} \   \rho(\cdot,0)=\rho(\cdot,1)=\text{\boldmath $\delta$}  \ \text{and}\\
&  |\beta|\leq A\rho^{1/2} \ \text{for some}  \ A \in L^2(\rt) \ \text{and} \
 \dfrac{\beta^2}{\rho} =0 \ \text{a.e. in} \ \{\rho=0\} \Big \}, 
\end{split}
\ee
\be\label{takis241.151}
\bs
\mathcal B=\Big\{ (\rho,\alpha):\; & \rho\in  C([0,1]; \mathcal P_2) \cap L^2(\rt):  \\[1mm]
&\pt_t \rho + \pt_x (\rho \alpha) =0,  \   \rho(\cdot,0)=\rho(\cdot,1)=\text{\boldmath $\delta$} \ \text{and} \ \int_0^1\int_\R \rho \alpha^2 dx dt <\oo \Big\}, 
\end{split}
\ee
\be\label{takis240}
\begin{split}
  \mathcal A_\eta=\Big \{& (\rho,\beta):\;  \rho \in C([0,1]; \mathcal P_2) \cap L^2(\rt) :\\
& \pt_t \rho + \pt_x \beta=0 \ \text{in} \ \rt  \ \text{and} \ \rho(\cdot,1)=\text{\boldmath $\delta $} \ \ \text{and}\\
 & |\beta|\leq A\rho^{1/2} \ \text{for some} \ A \in L^2(\rt) \ \text{and} \  \dfrac{\beta^2}{\rho} =0 \ \text{a.s. in} \ \{\rho=0\} \Big \},
\end{split}
\ee
and 
\be\label{takis240.151}
\bs
 \mathcal B_\eta=\Big \{& (\rho,\alpha):\;  \rho \in C([0,1]; \mathcal P_2) \cap L^2(\rt) :\\
& \pt_t \rho + \pt_x (\rho \alpha)=0 \ \text{in} \ \rt,   \ \rho(\cdot,1)=\text{\boldmath $\delta $} \  \text{and}
  \int_0^1 \int_\R \rho \alpha^2 dx dt<\oo \Big\}.
\end{split}
\ee
In the next two lemmata we collect two observations about $\mathcal A$, $\mathcal A_\eta$ and the energy functionals 
in \eqref{takis15} and \eqref{takis15.111}. Although their proofs are either immediate or known, we present them for the benefit of the reader.

\begin{lem}\label{lem.oxford17}
Any $\beta$ in either $\mathcal A$ or $\mathcal A_\eta$ is in $L^{4/3}(\rt)$. 
\end{lem}
\begin{proof}
Since $\ds \beta= \frac{\beta}{\sqrt{\rho}} \sqrt{\rho}$ and $\ds \frac{\beta}{\sqrt{\rho}} \in L^2(\rt)$, the conclusion follows from  Holder's inequality. 

\end{proof}

\begin{lem}\label{lem.oxford18}
The energy functionals in \eqref{takis15} and \eqref{takis15.111} are strictly convex.
\end{lem}
\begin{proof}
In view of the well known strict convexity of the energy functionals and the convexity equation satisfied by $\rho$ and $\beta$, to conclude we only need to show that  the constraints $|\beta|\leq \sqrt{\rho}$ and $\ds \frac{\beta^2}{\rho}=0$ a.s. in $\{\rho=0\}$  are both  convex.
\vs 
Fix  $A_1, A_2\in L^2(\rt)$. Then 
$\sqrt{\rho_1} A_1 + \sqrt{\rho_2} A_2 \leq 2\sqrt{\rho_1 + \rho_2} \max(A_1, A_2)$,
 and $\max(A_1, A_2)\in L^2(\rt)$.
\vs
Also,  since $\ds \{\frac{\rho_1+\rho_2}{2}=0\}\subset \{\rho_1=0\}\cup\{\rho_2=0\}$, $\ds \frac{\beta_1^2}{\rho_1 +\rho_2}\leq \frac{\beta^2_1}{\rho_1}$ and $\ds \frac{\beta_2^2}{\rho_1 +\rho_2}\leq \frac{\beta^2_2}{\rho_2}$, it follows that  $\ds \frac{\beta_1^2}{\rho_1 +\rho_2}=0$ in $\{\rho_1=0\}$ and  $\ds \frac{\beta_2^2}{\rho_1 +\rho_2}=0$ in $\{\rho_2=0\}$.

\end{proof}

\vs

The next proposition establishes that connection between the variational problems $\ol I_\eta$ and $\ol I$ and the Euler-Lagrange systems \eqref{takis200.0} and \eqref{takis2}, 

\begin{prop}\label{lem.oxford18}             
Let  $(\ol u_\eta, \ol \rho_\eta)$ and $(\ol u, \ol \rho)$ be the viscosity solutions 
to   \eqref{takis200.0}  and \eqref{takis2} constructed above. Then $(\ol \rho_\eta, \ol \rho_\eta \pt_x \ol u_\eta)$ and  $(\ol \rho, \ol \rho \pt_x \ol u)$ are the unique minimizers  of $\ol I_\eta$ and $\ol I$ respectively.
 \end{prop}
\begin{proof}
The uniqueness of minimizers of $\ol I_\eta$ and $\ol I$  is an immediate consequence of the strict convexity of the  energy functionals. 
\vs 

To show that  $(\ol \rho_\eta, \ol \rho_\eta \pt_x \ol u_\eta)$  is a minimizer of  $\ol I_\eta$, 
we pick any   $(\rho, \beta) \in \mathcal A_\eta$, use \eqref{phl1} and observe that, in view of the convexity of the map $(\rho, \beta) \to \dfrac{1}{2}(\rho^2 + \dfrac{\beta^2}{\rho})$, we have, for $\ol \beta_\eta =\ol \rho_\eta \pt_x\ol u_\eta$, 
\[ \bs \ds\int_0^1\ds\int_\R &\dfrac{1}{2}\big[\rho^2 + \dfrac{\beta^2}{\rho}\big] dx dt + \ds \int_\R \rho (x,0) \dfrac{x^2}{2\eta}  \geq \ds\int_0^1\ds\int_\R \dfrac{1}{2}[{\ol \rho_\eta}^2 +\dfrac{{\ol \beta_\eta}^2}{\ol \rho_\eta} \big]dx dt + \ds\int_\R  \ol \rho_\eta(x,0) \dfrac{x^2}{2\eta} dx \\[1.2mm]
& +  \ds\int_0^1\ds\int_\R \big[(\rho -\ol \rho_\eta)\ol \rho_\eta + (\beta-\ol \beta_\eta) \dfrac{\ol \beta_\eta}{\ol \rho_\eta} - \dfrac{1}{2}\dfrac{\ol \beta_\eta^2}{\ol \rho_\eta^2}(\rho-\ol \rho_\eta)  \big] dx dt + \ds \int_\R (\rho-\ol \rho_\eta)(x,0) \dfrac{x^2}{2\eta} dx\\
& = \ds\int_0^1\ds\int_\R \big[(\rho -\ol \rho_\eta)\ol \rho_\eta + (\beta- \ol \rho_\eta \pt_x \ol u_\eta)\pt_x \ol u_\eta - \dfrac{1}{2}(\rho-\ol \rho_\eta)(\pt_x\ol u_\eta)^2 \big]dx dt + \ds \int_\R (\rho-\ol \rho_\eta)(x,0) \dfrac{x^2}{2\eta} dx \\
& = \ds\int_0^1\ds \int_\R \big[\pt_t \ol u_\eta (\rho-\ol \rho_\eta) + \ol u_\eta \pt_t (\rho-\ol \rho_\eta) \big]dx dt + \ds \int_\R (\rho-\ol \rho_\eta)(x,0) \dfrac{x^2}{2\eta} dx\\
& = \ds\int_0^1\ds \int_\R \pt_t\big[ \ol u_\eta (\rho-\ol \rho_\eta) \big]dx dt + \ds \int_\R (\rho-\ol \rho_\eta)(x,0) \dfrac{x^2}{2\eta} dx=0. 
\end{split}
\]

%
\vs 
The claim about the minimizing property of  $(\ol \rho, \rho \pt_x\ol u)$ follows as the one for $(\ol \rho_\eta, \rho_\eta \pt_x \ol u_\eta)$ the only difference being that  we pick $(\rho, \beta) \in \mathcal A$ and, instead of integrating in time from $0$ to $1$, we now integrate over $(h,1)$.

\end{proof}

\vs

\subsection*{The uniqueness of solutions to the singular planning problem.} To prove that the viscosity solution $(\ol u,\ol \rho)$ is the unique solution to the planing system \eqref{takis2}, we need  to specify the regularity of the solutions we are considering. The motivation for these requirements is that they are satisfied by the viscosity solution constructed above.
\vs

\newtheorem*{definition}{Definition}
\begin{definition}
A pair $(u, \rho) \in C(\R\times (0,1])\times\big(C([0,1];{\mathcal P}_2)\cap L^1(\rt)\big)$ is a solution to \eqref{takis2}, if $u$ is a viscosity solution to the Hamilton-Jacobi equation and $\rho$ is a distributional solution to the continuity equation and, in addition, 
\be\label{takis40.110}
\bs
& (i)~u(x,t)\geq x^2/2t, \\ 
& (ii)~\int_\R u(x,t) \ol \rho(x,t) dx \leq \int_\R \ol \rho(x,t) \dfrac{x^2}{2t} dx + \text{o}(1) \ where  \ 
\underset{t\to 0} \lim o(1)=0,\\ 
& (iii)~u\in C^{0,\gamma}(\R\times (\theta,1]) \ \ \text{for some} \ \gamma \in (0,1) \ \text{and for all} \ \theta\in (0,1)  \ \ \text{and}\\ 
& (iv)~\ds\int_0^1 \ds\int_\R [\rho^2 + \rho (\pt_x u)^2] dx dt<\oo.
\end{split}
\ee
\end{definition} 
\vs 
Before we state and prove the uniqueness result for the planning problem, we present the following auxiliary fact.
\vs
%
%
%
%
%
%
%
%

\vs
\begin{thm}\label{thm10}
Let $(\rho, u)$ be a solution of \eqref{takis2} satisfying \eqref{takis40.110}. 
Then 
$\rho=\ol \rho$ and $u=\ol u$.
\end{thm}
\begin{proof}
As it is common when studying systems with mean field games structure, the uniqueness is based on a  classical computation which gives, for any $\theta\in (0,1)$, the identity 
\be\label{takis35}
\begin{split}
 \int_\R &(\rho-\ol \rho)(x,1-\theta) (u-\ol u)(x,1-\theta) dx -  \int_\R (\rho-\ol \rho)(x,\theta) (u-\ol u)(x,\theta) dx \\[1.5mm]
& =\int_\theta^{1-\theta} \int_\R \Big [ (\rho -\ol \rho)^2 + \dfrac{1}{2} (\rho + \ol \rho) (\pt_x u -\pt_x \ol u)^2  \Big ] dx dt.
 \end{split}
 \ee
\vs
Assuming  we can show that
\be\label{takis36.1}
\underset{\theta\to 0} \lim  \int_\R (\rho-\ol \rho)(x,1-\theta) (u-\ol u)(x,1-\theta) dx\leq 0 \ \text{and} \ 
 \underset{\theta\to 0}\lim \int_\R (\rho-\ol \rho)(x,\theta) (u-\ol u)(x,\theta) dx\geq 0,
\ee
\vs
it follows that 
\[\int_0^{1} \int_\R \Big [ (\rho -\ol \rho)^2 + \dfrac{1}{2} (\rho + \ol \rho) (\pt_x u -\pt_x \ol u)^2  \Big ] dx dt=\underset{\theta \to 0}\lim \int_\theta^{1-\theta} \int_\R \Big [ (\rho -\ol \rho)^2 + \dfrac{1}{2} (\rho + \ol \rho) (\pt_x u -\pt_x \ol u)^2  \Big ] dx dt=0,\]
\vs
which in turn implies that 
\[ \rho=\ol \rho \ \ \text{a.e. in} \ \ \rt \ \ \text{and} \ \ \ \pt_x u =\pt_x \ol u \ \ \text{a.e. in} \ \  {\mathcal O}. \]
Moreover, since  $u$ and $\ol u$ solve the same equation in $\mathcal O$,   we also have that 
\[\pt_t u=\pt_t \ol u \ \ \text{a.e.  in} \ \  {\mathcal O}\]
and, hence, in view of the continuity of $u$ and $\ol u$, that, for some constant $c\in \R$, 
\be\label{takis3000.0}
u-\ol u =c \ \ \text{on} \ \ \ol {\mathcal O}.
\ee

It follows from \eqref{takis40.110} and Proposition~\ref{prop31} that 
\[\underset{t\to 0} \lim \int_\R \ol \rho (x,t) u(x,t) dx= \underset{t\to 0}  \lim \int_\R \ol \rho (x,t)  \ol u(x,t) dx=0,\]
and, thus, 
\begin{equation}\label{takis1.111}
\underset{t\to 0} \lim \int_\R \ol \rho (x,t)[ u(x,t) -  \ol u(x,t)] dx=0.
\end{equation}
Since $\ol \rho$ is supported in $\ol {\mathcal O}$ and $\int_\R \ol \rho(x,t) dx=1$,  \eqref{takis3000.0} and \eqref{takis1.111}  yield that $c=0$, that is
\be\label{takis3000}
u=\ol u  \ \ \text{on} \ \ \ol {\mathcal O}.
\ee

To conclude, we need to show that $u=\ol u$ on $R\times(0,1]$. This follows from the fact that both $u$ and $\ol u$ are viscosity solutions of the Hamilton-Jacobi equation of the system, and, hence, satisfy the same variational formula (see Proposition~\ref{prop1.0}). 

\vs
We know return to the $\theta\to 0$ limits in \eqref{takis36.1} and begin with the behavior near $t=0$. 
\vs
In view of the signs of $\ds \int_\R\rho (x,\theta)u(x,\theta)dx$ and $\ds \int_\R\ol \rho (x,\theta)\ol u(x,\theta)dx$, it suffices to show that 
\[\underset{\theta\to 0} \lim  \int_\R \rho(x,\theta)\ol u(x,\theta) dx=0 \ \ \text{and} \ \ \underset{\theta\to 0} \lim  \int_\R \ol \rho(x,\theta) u(x,\theta) dx=0.\]
\vs

The first limit follows from \eqref{phl31.1} in Proposition~\ref{prop31}, while 
while the second follows from  \eqref{phl31} and \eqref{takis40.110} (ii).

\vs

We turn now to the behavior near $t=1$. In view of \eqref{phl1} in Proposition~\ref{phl25}, we only need to show that 
\be\label{takis120}
\underset{h\to 0}\lim \int_\R (\rho-\ol \rho)(x,1-h) u(x,1-h)dx=0.
\ee

We show that 
\be\label{takis174}
\underset{h\to 0} \lim \ds \int_\R \rho(x,1-h)u(x,1-h) dx= u(0,1), 
\ee
and 
\be\label{takis175}
\underset{h\to 0} \lim \ds \int_\R \ol \rho(x,1-h)u(x,1-h) dx= u(0,1).
\ee

For \eqref{takis174}, we use the identity
\vskip.025in
\be\label{uc1}
\dfrac{d}{dt} \ds \int_\R \rho(x,t) u(x,t) dx= \ds \int_\R \big[\rho^2(x,t) + \dfrac{1}{2} \rho(x,t) (\pt_x u)^2(x,t)\big] dx,
\ee
\vskip.075in
and the continuity of $u$ near $t=1$ to get 
\[ \bs
u(0,1) -  \ds \int_\R \rho(x,1-h) u(x,1-h) dx&= \ds \int_\R \rho(x,1) u(x,1) dx - \ds \int_\R \rho(x,1-h) u(x,1-h) dx\\[1.2mm]
&= \int_{1-h}^1\ds \int_\R \big[\rho^2(x,t) + \dfrac{1}{2} \rho(x,t) (\pt_x u)^2(x,t) \big]dxdt,
\end{split}
\]
and we conclude from the fact that
\[\underset{\theta\to 0} \lim \int_{1-\theta}^1\ds \int_\R \big[\rho^2(x,t) + \dfrac{1}{2} \rho(x,t) (\pt_x u)^2(x,t) \big]dxdt=0.\]

For the proof of \eqref{takis175},  the form of $\ol \rho$ yields that 
\[
\ds \int_\R \ol \rho(x,t) u(x,t) dx - u(0,1) = \ds \int_\R \ol \rho(x,t) [u(x,t) -u(0,1)] dx =\ds \int_\R (1-x^2)_+ [u(l(t) x,t) -u(0,1)] dx.
\] 
It then follows from  the continuity of $u$ that 
\[\underset{t \to 1} \lim \ds \int_\R (1-x^2)_+ (u(l(t) x,t) -u(0,1)) dx=0.\]

%

\end{proof}

\section{uniform estimates for the solution to  the approximate viscous problem}

This section is the core of the analysis in the paper, since  we obtain here estimates on the solution  $(u_{\ep,\eta}, \rho_{\ep,\eta})$ to the viscous system  \eqref{takis1} which are uniform in the diffusivity $\ep$ and the relaxation parameter $\eta$ used in the narrow edge initial condition at $t=0$ for the potential $u_{\ep,\eta}$. 
These bounds yield the necessary compactness to pass to  the $\ep,\eta \to 0$  limit to show the convergence in $\rt$  of solution to  $(u_{\ep,\eta}, \rho_{\ep,\eta})$ to \eqref{takis1}  to the viscosity solution $(\ol u,\ol \rho)$ to \eqref{takis2} for all $\ep,\eta>0$ satisfying \eqref{uc0}. We note that the exponent $\kappa$ in \eqref{uc0} is not fixed but it depends on the particular family  $(\ep ,\eta)$. 

%
%

\vs
We provide next  a brief outline of the path leading to uniform estimates that will allow to pass in the limit to show the local uniform convergence of the $u_{\ep,\eta}$'s and the $L^2-$ convergence of the $\rho_{\ep,\eta}$'s. 
\vs


In Proposition~\ref{prop10} we obtain a, uniform in $\ep$ and $\eta$ satisfying \eqref{uc0},  bound for  the minimization problem $I_{\ep,\eta}$, which  is critical   for the rest of the paper, since it yields an upper bound for the total energy of \eqref{takis1}.  
\vs
Proposition~\ref{prop11.1} provides a number of important estimates for some key functionals like the total momentum   $\ds \int_\R \rho_{\ep,\eta} u_{\ep,\eta} dx$ and the second-order moment  of $\rho_{\ep,\eta}$.
\vs
Proposition~\ref{prop12} is about controlling  in  $L^2_{\text{loc}}$ up to time $t=1$  the velocity $\pt_x u_{\ep,\eta}$, while Proposition~\ref{prop125} provides uniform estimates for $\sqrt{\rho_{\ep,\eta}} \pt^2_x u_{\ep,\eta}$ and $\pt_x \rho_{\ep,\eta}$ in $L^2 (\R\times (\theta, 1-\theta))$ for any $\theta \in (0,1)$.  
\vs
The next step in Proposition~\ref{prop15} is to obtain an integrable control on the $C^{0,\alpha}-$norm of $\rho_{\ep,\eta}$ near $t=1$, which  in  Proposition~\ref{prop10} gives a uniform estimate for the $C^{0,\alpha}-$norm of $u_{\ep,\eta}$ up to time $t=1$. 
\vs
The final important estimate in Proposition~\ref{prop20.0} is an upper bound, as $t\to 0$,  in a distributional sense for $u_{\ep,\eta}/(2t)$, a fact that  is used to  that the effect the  diffusion near $t=0$  is negligible.  
\vs

\vs
For definiteness, we recall the basic existence and uniqueness result of solutions to the system \eqref{takis1}.

\begin{prop}\label{prop0}
For each $\ep, \eta>0$, \eqref{takis1} has a unique solution solution $(u_{\ep,\eta}, \rho_{\ep,\eta}) \in C(\R\times [0,1]) \times C([0,1]; \mathcal P)$
which is classical in $\rt$. 
\end{prop}

%

%
%
We turn now to  the variational characterization of the solution as the minimizer of 

\be\label{takis61} 
I_{\ep,\eta}=\underset{(\rho,\beta) \in \mathcal A_{\ep,\eta}} \inf \Big\{\int_\R \rho(x,0) \dfrac{x^2}{2\eta} dx +\int_0^1\int_\R \big[\rho^2(x,t) +\dfrac{1}{2} \dfrac{\beta^2 (x,t)}{\rho(x,t)}\big] dx dt\Big \}
\ee

where 
\be\label{takis62}
\bs
\mathcal A_{\ep,\eta}=\Big\{ (\rho,\beta):\;  &\rho\in  C([0,1]; \mathcal P) \cap L^2(\rt):  \\
& \pt_t \rho +\ep \pt_x^2\rho+ \pt_x \beta=0 \ \text{and} \   \rho(\cdot,1)=\text{\boldmath $\delta$}, \\
&  |\beta|\leq A\rho^{1/2} \ \text{for some}  \ A \in L^2(\rt) \ \text{and} \
 \dfrac{\beta^2}{\rho} =0 \ \text{a.e. in} \ \{\rho=0\} \Big \}.
\end{split}
\ee

  
\vs

Since  the integrand inside the infimum  in \eqref{takis62} is strictly convex in $(\rho,\beta)$   and the constraint in $\mathcal A_{\ep,\eta}$ is convex, there exists a unique minimizer $(\rho_{\ep,\eta}, \beta_{\ep,\eta})$, which is 
the solution to the Euler-Lagrange system  \eqref{takis1}. 
\vs

Finally, we remark that the minimization problem can also be written as
\be\label{takis63.1}
I_{\ep,\eta}=\underset {\mathcal B_{\ep,\eta}} \inf \Big \{  \int_\R \rho(x,0) \dfrac{x^2}{2\eta} dx +\int_0^1\int_\R \big[\rho^2 +\dfrac{1}{2} \rho \alpha^2 \big](x,t) dx dt \Big\},
\ee
where
\[
\bs \mathcal B_{\ep,\eta} =\Big \{ (\rho, \alpha):& \;   \rho \in C([0,1]; \mathcal P) \cap L^2(\rt):   \sqrt{\rho} \alpha \in  L^2(\rt), \\
&  \pt_t \rho +\ep \pt_x^2\rho+ \pt_x  (\rho \alpha)=0 \ \text{in} \ \rt  \ \text{and} \ \rho(\cdot,1)=\delta \Big \}.
\end{split}
\]

Before we begin with the main part of this section, we expand here the brief discussion in the introduction about the fact that the introduction of the diffusion does not allow for solutions to 
\[\bs
& \pt_t u_{\ep, \eta} -{\ep} \pt^2_{x}u_{\ep, \eta} + \dfrac{1}{2}(\partial_x u_{\ep, \eta})^2 = \rho_{\ep, \eta}  \; \; \text{in} \; \; \R\times (0,1),\\[1.2mm] 
& \pt_t \rho_{\ep, \eta} + \ep \pt^2_{x} \rho_{\ep, \eta}+\pt_x( \rho_{\ep, \eta}\pt_x u_{\ep, \eta})=0 \;\; \text{in} \;\; \R\times (0,1) \ \ \text{and} \ \  \rho_{\ep, \eta}(\cdot,1)= \rho_{\ep,\eta}(\cdot,0)=\text{\boldmath $ \delta$}.
\end{split}
\]
a fact that, in view of the connection between the solutions to the Euler-Lagrange system and the variational problem, is equivalent to saying that 
\[I_{\ep,0}=\infty.\]
In view of this last comment,  we show next that it is not possible to have solution to the sde 
\be\label{paris1}
dX_t=\sqrt{2\ep} dB + a(X_t,t) dt \ \text{in} \ (0,1) \ \ \text{and} \ \ X_0=X_1=0,
\ee
for a Brownian motion $B$ and a drift  $a\in L^2_{\omega, t}$, where $L^2_{\omega, t}$ stands for the space $L^2(\Omega \times (0,1))$. Note that this last condition on $a$ is equivalent to saying that, if $\rho$  is the law of the solution to \eqref{paris1}, then $\int_0^1\int_\R \rho a^2(x,t) dx dt<\oo$.
\vs
The proof of the claim above follows using Girsanov's transformation which, roughly speaking, changes the solution to \eqref{paris1} to another  Brownian motion $Z$ with $Z_0=Z_1=0$. 
\vs 
To make this argument possible, we need to assume that the drift is bounded. So instead of solving \eqref{paris1} for a given $a\in L^2_{\omega, t}$, we consider 
a sequence $a_n \in L^\oo_{\omega,t}$ of drifts, for example, $a_n=\min\big(\max (a,-n),n\big )$, such that, as $n\to \oo$,  $a_n \to a$ in $L^2_{\omega, t}$. Then the solutions to the sde
\be\label{paris2} dX_{n,t}=\sqrt{2\ep} dB + a_n(X_{n,t},t) dt \ \text{in} \ (0,1) \ \ \text{and} \ \ X_{n,0}=X_{n, 1}=0\ee
converge to $X$ in $C_t(L^2_\omega)$. 
Working with $X_n$'s we can show that, if there were such solutions to \eqref{paris2}, then $\E\int_0^1 a^2_n \to \oo$, which contradicts the fact that $a\in L^2_{\omega, t}$.
\vs
Having explained how to make the argument rigorous, for the shake of brevity, we argue next as if the drift $a$ in \eqref{paris1}
had the necessary properties to use the Girsanov  transformation and obtain a contradiction.
\vs
Assuming that there is a solution to \eqref{paris1}, we find that we must have 
\be\label{paris3}
\E\big[Z^2_1 \exp\big( \sqrt{2\ep}\int_0^1 a dB -\ep \int_0^1 a^2 ds\big) \big]=0.
\ee
Since the Brownian motion $Z$ has the property that almost surely $Z_1\neq 0$,  it follows from \eqref{paris3} that we must have,  again almost surely, $\exp\big( \sqrt{2\ep}\int_0^1 a dB -\ep \int_0^1 a^2 ds\big)=0$, and, thus, 
$ \ep \int_0^1 a^2 ds - \sqrt{2\ep}\int_0^1 a dB =\oo \ \text{almost surely}, $
which implies  that we must have 
$\E \int_0^1 a^2 =\oo,$
which contradicts the key property of the drift.

\vs

It is clear that, if we could prove that, as $(\ep, \eta)\to 0$, $I_{\ep,\eta}\to \ol I$, where $\ol I$ is given by \eqref{takis15}, 
we would have  that the minimizer $(u_{\ep,\eta}, \rho_{\ep,\eta})$ converges to the unique minimizer of $\ol I$.  
Here, however, we work the other way around, that is, we first show the latter and then  the former. 
\vs


\vs
We begin with a result about a uniform upper bound on $I_{\ep, \eta}$ as well as the easy part of what would have been a convergence result.

\begin{prop}\label{prop10} Assume \eqref{uc0}. 
Then
\be\label{takis63.1} 
\underset{\ep\to 0,\eta\to 0, \ep(-\log \eta)^\kappa \to 0} \lim \sup \; I_{\ep,\eta}<\oo 
 \ \ \text{and} \ \ 
\ol I \leq \underset{\ep\to 0,\eta\to 0, \ep(-\log \eta)^\kappa \to 0} \liminf I_{\ep,\eta}.
\ee
\end{prop}
\begin{proof}
To prove  the first claim,  we construct an admissible pair $(\rho_{\ep,\eta}, \beta_{\ep,\eta})$
such that
 \[ I_{\ep,\eta}(\rho_{\ep,\eta}, \beta_{\ep,\eta})\lesssim 1.\] 
 
 \vs
 
Let  $p$ be the fundamental solution of the heat equation $\pt_t w-\pt^2_x w=0$ in $\rt$, that is, $p(x,t)=\exp{(-x^2/4t)}/(2\sqrt{\pi t})$.
 \vs
 
 In $\R\times (1/2,1)$, we take $\rho(x,t)=p(x,b(t))$ with $b$ a smooth  decreasing function and $b(1)=0$.
 It is  immediate that 
 \[0=\pt_t \rho -\dot b\pt_x^2 \rho= \pt_t \rho +\ep \pt_x^2 \rho - (\dot b+\ep) \pt_x (\rho  \pt_x (\log \rho))=\pt_t \rho +\ep \pt_x^2 \rho + \pt_x (\rho \alpha)  \ \ \text{with} \ \ \alpha (x,t)= \dfrac{\dot b(t) +\ep}{2 b(t)}x. \]
\vs

We also have 
\[ \int_{1/2}^1 \int_\R \rho^2 dxdt \approx \int_{1/2}^1 \dfrac{1}{\sqrt{b}} dt \ \ \text{and} \ \ 
\int_{1/2}^1 \int_\R \rho \alpha^2 dx dt \approx \int_{1/2}^1\dfrac{(\dot b+\ep)^2}{b} dt.\]
\vs

Let  $b(t)=\ep (1-t) + (1-t)^\zeta$ for $\zeta \in (1,2)$. Then, 

\[\int_{1/2}^1 \dfrac{1}{\sqrt{b}} dt \leq \int_0^{1/2} t^{-\zeta/2} dt \lesssim 1 \ \text{ and} \ \int_{1/2}^1\dfrac{(\dot b+\ep)^2}{b} dt\leq \int_0^{1/2} t^{\zeta-2} dt \lesssim 1. \]

We note that $\lesssim$ above depends on $\zeta$ but, since $\zeta$ is chosen once and for all, we omit this fact.
\vs

Finally,

\[\rho(x,1/2)=\dfrac{1}{2\sqrt{\pi b_\ep}}\exp[-\dfrac{x^2}{4b_\ep}] \ \ \text{with} \ \ b_\ep=\ep/2 +1/2^\theta.\]
\vs
In $(0,1/2)$, we take $\rho(x,t)=p(x, \eta +c(t))$ with  $c(t)=c_0 t /(-\log t)^{\mu}$ with $\mu>1$.  
It follows as above that 

\[\pt_t \rho + \ep \pt_x^2 \rho + \pt_x(\rho \alpha)=0 \  \text{in} \ \rt \ \ \text{with} \ \  \alpha (x,t)=\dfrac{\dot c(t) +\ep}{\eta + c(t)} x.\]

At $t=0$, we have 
\[\int_\R \frac{x^2}{2\eta}\rho(x,0)dx =\int_\R \frac{x^2}{2\eta} \dfrac{1}{2\sqrt{\pi \eta}} \exp{(-x^2/4\eta)} dx  \approx 1,\] 
while, at $t=1/2$,
\[\rho(x,1/2)=\dfrac{1}{2\sqrt{\pi c_\ep}}\exp[-\dfrac{x^2}{4c_\ep}] \ \ \text{with} \ \ c_\ep=\eta + c_0/2(-\log1/2)^\mu.\]
For  $\rho$ to be continuous at $t=1/2$ we need to choose $c_0=c_0(\ep,\eta)$ so that $b_\ep=c_\ep$, that is,
\[\ep/2 + 1/2^\theta=\eta +c_0/2(log 2)^\mu;\]
note that
\[\underset{\ep\to 0, \eta\to 0}\lim c_0=\frac{2(\log 2)^\mu}{2^\theta}.\]

Moreover,

\[\int_0^{1/2}  \int_\R  \rho^2 dxdt \approx \int_{0}^{1/2} \dfrac{1}{\sqrt{\eta + c(t)}} dt \ \ \text{and} \ \ \int_{0}^{1/2}  \int_\R \rho \alpha^2 dx dt \approx \int_{0}^{1/2} \dfrac{(\dot c(t) +\ep)^2}{\gamma +c(t)} dt.\]
 \vs
 
%
%

In view of the choice of $c$, we have  
\[\ds \int_0^{1/2}\dfrac{1}{\sqrt{\eta + c(t)}}dt= \ds \int_0^{1/2}\dfrac{1}{\sqrt{\eta +  \dfrac{c_0 t}{(-\log t)^\mu}} }dt\lesssim \int_0^{1/2}\dfrac{(-\log. t)^{\mu/2}}{t^{1/2}} dt\lesssim 1, \]
and 
\[\displaystyle \int_{0}^{1/2} \dfrac{(\dot c(t) +\ep)^2}{\eta +c(t)} dt \leq 
2 \int_{0}^{1/2} \dfrac{(\dot c(t))^2}{\eta + c(t)} dt + 2 \int_{0}^{1/2} \dfrac{\ep^2}{\eta + c(t)} dt.\]
\vs
Since $\mu>1$, it follows that   
\[ \ds \int_{0}^{1/2} \dfrac{(\dot c)^2}{\eta + c(t)} dt \ds \int_{0}^{1/2} \dfrac{1}{t(-\log t)^\mu} \lesssim 1.\]

The estimate for the last term is a bit more involved requiring to split the integration over $(0,1/2)$ to to  integrals over $(0,\eta)$ and $(\eta, 1/2)$, that is,
\[\ds \int_0^{1/2} \dfrac{\ep^2}{\eta+c(t)} dt =\ds \int_0^{\eta} \dfrac{\ep^2}{\eta+c(t)} dt +\ds \int_\eta^{1/2} \dfrac{\ep^2}{\eta+c(t)} dt.\]
The upper bound for  the first integral is immediate since
\[\ds \int_0^{\eta} \dfrac{\ep^2}{\eta+c(t)} dt\leq \ds \int_0^{\eta}\dfrac{\ep^2}{\eta} dt\leq \ep^2.\]

For the second integral, we use that  $-\log t\leq -\log \eta$ in $(\eta,1/2)$ and set  $A_\eta=1/(2\eta (-\log\eta)^\mu)$ to find  
\begin{equation*}
\begin{split}
 \ds \int_\eta^{1/2} \dfrac{\ep^2}{\eta+c(t)} dt & \leq \ds \int_\eta^{1/2} \dfrac{\ep^2 (-\log t)^\mu}{\eta (-\log \eta)^1\mu +t} dt
\leq \ds \int_0^{A_\eta} \dfrac{\ep^2 (-\log \eta)^\mu}{1+s} ds\\
& \lesssim \ep^2 (-log \eta)^\mu \log \big(\dfrac{1}{2\eta (-\log \eta)^\mu} ) \lesssim 1,
\end{split}
\end{equation*}
with the last estimate following from \eqref{uc0} choosing $\mu=\kappa +\tau$ for some small $\tau$. 
\vs

In conclusion, we have an admissible  pair $(\rho,\alpha)$ which provides a uniform upper bound for the energy for all  $(\ep,\eta)$ satisfying \eqref{uc0}.
\vs

%
To prove \eqref{takis63.1}(ii), let $(u_{\ep,\eta}, \rho_{\ep,\eta})$ be the minimizer of 
$I_{\ep,\eta}$, that is, for $\beta_{\ep,\eta}=\rho_{\ep,\eta} \pt_x u_{\ep,\eta}$, 
\be\label{takis64}
I_{\ep,\eta}=\int_\R \rho_{\ep,\eta}(x,0)\dfrac{x^2}{2\eta} +
\int_0^1\int_\R \big[\rho^2_{\ep,\eta} + \dfrac{1}{2} \dfrac{\beta^2_{\ep,\eta}}{\rho_{\ep,\eta}}\big](x,t)dx dt,
\ee
and
\be\label{takis65}
\pt_t \rho_{\ep,\eta} +\ep \pt_x^2 \rho_{\ep,\eta} + \pt_x \beta_{\ep,\eta} =0 \ \text{in} \ \R\times(0,1) \ \text{and} \ \rho_{\ep,\eta}(\cdot,1)=\text{\boldmath $\delta$}.
\ee
\vs

It follows  from \eqref{takis63.1}(i) that 
\be\label{takis66}
\int_0^1\int_\R \rho^2_{\ep,\eta} dx dt \lesssim 1 \ \ \text{and} \ \ \int_0^1\int_\R  \dfrac{\beta^2_{\ep,\eta}}{\rho_{\ep,\eta}}dx dt \lesssim 1.
\ee
\vs
We also know  that, for any measurable subset $A$ of $\R\times(0,1)$,
\[\iint_A |\beta_{\ep,\eta}| dx dt \lesssim 1,\]
which is an immediate consequence of \eqref{takis62} and the elementary inequality 
\[|\beta |=\dfrac{|\beta |\rho}{\rho} \leq \dfrac{\beta^2}{2\rho} + \dfrac{\rho}{2}.\]
It follows that the $\rho_{\ep,\eta}$'s converge weakly in  $L^2(\R\times(0,1))$ to some $\rho$, and, in view of their bound in $L^{4/3}$, the $\beta_{\ep,\eta}$'s converge weakly in $L^1(\R\times(0,1))$ to some $\beta$. Finally, using  the lower semicontinuity of the map $\ds (\rho,\beta) \to \int_0^1\int_\R \frac{\beta^2}{\rho} dx dt$, we find 
\[\underset{\ep\to 0,\eta\to 0, \ep(-\log \eta)^\kappa \to 0} \liminf I(\rho_{\ep, \eta}, \beta_{\ep, \eta}) \geq I(\rho, \beta),\]
and 
\[\pt_t\rho +\pt_x \beta=0 \ \ \text{in} \ \ \R\times (0,1) \ \ \text{and} \ \  \rho(\cdot,1)=\text{\boldmath $\delta$},\]
and, hence, the claim.

\end{proof}

All the estimates in the rest of this section rely on the a priori bound of energy in Proposition~\ref{prop10}.
\vs

To prove directly the convergence of the variational problems we need to establish the opposite inequality using $\limsup$ in place of $\liminf$. We  do not have a direct proof for this.  Instead we show  directly the convergence using the convergence of the minimizers.
\vs

Next we  use a key identity, the fact that $(u_{\ep,\eta}, \rho_{\ep,\eta})$ is the minimizer of $I_{\ep,\eta}$ and  Proposition~\ref{prop10}  to provide several key estimates  or  $u_{\ep,\eta}$ and $\rho_{\ep,\eta}$, which are uniform for  $\ep$ and $\eta$ satisfying \eqref{uc0}. 
The estimates will be sharpened in Proposition~\ref{prop20.0}.
\vs
\begin{prop}\label{prop11.1}
Uniformly in $\eta$ and $\eta$ satisfying \eqref{uc0}, we have 
\begin{equation}\label{takis23}
\bs
&\ds \int_0^1 \ds \int_\R [\rho_{\ep,\eta}^2(x,t) + \dfrac{1}{2}\rho_{\ep,\eta}(x,t) (\pt_x u_{\ep,\eta})^2(x,t)] dx dt  +{\displaystyle \int_\R} \rho_{\ep,\eta}(x,0) \frac{x^2}{2\eta} dx \lesssim 1 \ \ \text{and}\\[1.5mm]
&  + \ds \int_0^1 \ds \int_\R \rho_{\ep,\eta} \frac{x^2}{(t+\eta)^2} dx dt + \sup_{t\in[0,1]}[{\displaystyle \int_\R} \rho_{\ep,\eta}(x,t) u_{\ep,\eta}(x,t) dx +  {\displaystyle \int_\R} \rho_{\ep,\eta}(x,t) \dfrac{x^2}{\eta+t} dx ]  \lesssim 1.
\end{split}
\end{equation}
\end{prop}
\vs
\begin{proof}
Throughout the proof we omit the explicit dependence of $u_{\ep,\eta}$ and $\rho_{\ep,\eta}$ on ${\ep,\eta}$ and simply write $u$ and $\rho.$ Moreover, all estimates obtained below are uniformly in $\eta$ and $\eta$ satisfying \eqref{uc0}
\vs
The first three bounds are immediate from Proposition~\ref{prop10} and the fact that $(u_{\ep,\eta}, \rho_{\ep,\eta})$ is the minimizer of $I_{\ep,\eta}$. 

\vs

%
%
    %
%
The uniform in $t$, $\ep$ and $\eta$ bound on the total momentum ${\displaystyle \int_\R} \rho (x,t) u(x,t)dx$ follows  from the first two bounds  and the key identity
\be\label{takis24}
\dfrac{d}{dt} {\displaystyle \int_\R} \rho(x,t) u(x,t) dx= {\displaystyle \int_\R} [\dfrac{1}{2} \rho (x,t) (\pt_x  u)^2(x,t) dx + \rho^2(x,t)] dx.
\ee
\vs
For the final bound, we note that elementary computations yield
\be\label{ph10}
\bs
\frac{d}{dt} \ds \int_\R \rho \frac{x^2}{t+\eta} dx&=-\ds \int_\R \rho \frac{x^2}{(t+\eta)^2} dx 
+ \frac{2}{t+\eta} \ds\int_\R \rho (\pt_x u)  x dx 
-\frac{2\ep}{t+\eta} \ds \int_\R \rho(x,t) dx\\
& \leq -\frac{1}{2} \ds \int_\R \rho \frac{x^2}{(t+\eta)^2} dx + 2\ds\int_\R \rho (\pt_x u)^2dx.
\end{split}
\ee
The final estimate follows from \eqref{ph10}  and the fact that, in view of the first bound,
\[ \int_0^1  \int_\R \rho  (\pt_x u)^2 dx dt +  \int_\R \rho(x,0)\frac {x^2}{\eta} dx \lesssim 1. \]
Of course, to make \eqref{ph10} rigorous, we need to truncate $x^2$ by $(x\wedge R)^2$ and then let $R\to \oo$. 

\end{proof}



The next proposition improves the $L^2-$bound on $\sqrt{\rho_{\ep,\eta}} \pt_x u_{\ep,\eta}$ in \eqref{takis23} and provides  uniform control up to  time $t=1$  and away from $t=0$ for $u_{\ep,\eta}$ in  $L^\oo_t(L^1_{\text{loc},x})$ and for $\pt_x u_{\ep,\eta}$ in $L^2_t(L^2_{\text{loc},x})$. All estimates are away from $t=0$ since we want them to be independent of $\eta$.
\vs
We obtain an upper bound using comparison with some explosive solution. We believe that this method will be  useful in other settings, and, following the referee,  we call it ``coming down from infinity argument''.
 
\vs
 

\begin{prop}\label{prop12}
Uniformly in $\ep$ and $\eta$ satisfying  \eqref{uc0}, there exists $R_0\geq 1$ depending  only on the upper bounds in \eqref{takis23} such that,  for each $h \in (0,1/2)$ and $R\geq R_0$, there exists $C_{h,R}>0$ such that

\be\label{ma1}
\|\pt_xu_{\ep,\eta}\|^2_{L^2((-R,R)\times (h,1))} + \underset{t\geq h}\sup \|u_{\ep,\eta}\|^2_{L^2((h,1); H^1((-R,R)))} + \|u_{\ep,\eta}\|^2_{L^2((h,1); H^1((-R,R)))} \leq C_{h,R}
\ee
\end{prop}

Since the proof is long, we present some of the preliminary steps next as separate lemmata.
\vs

We begin the fact  the  mass of the $\rho_{\ep,\eta}$'s can be assumed to be in an interval $(-R,R)$ provided $R$ is uniformly large.

\begin{lem}\label{brod1}
There exists some $R_0\geq 1$, which is uniform in $\ep$ and $\eta$ satisfying  \eqref{uc0}, such that, for all $R\geq R_0$ and $t\in (h,1)$,
\be\label{takis30} 
\ds \int_{|x|\leq R} \rho_{\ep,\eta}(x,t)  dx  \geq \dfrac{1}{2}. 
\ee
\end{lem}

\begin{proof}

To simplify the notation, in what follows  we write $\rho$ instead of  $\rho_{\ep,\eta}$, and we denote by  $C$ constants that may change from line to line but  are uniform in $\ep$ and $\eta$ and may depend on $h$.  Finally, all estimates below will be uniform in $\ep$ and $\eta$. 
\vs

We begin with the observation  that, since  $t\in (h,1)$,  Proposition~\ref{prop11.1} yields 
that
\be\label{ph140.0}
\underset{t\geq h}\sup \ds\int_\R \rho(x,t) x^2 dx \lesssim_h 1.
\ee
It follows that there exists some $R_0\geq 1$ such that, for all $R\geq R_0$ and $t\in (h,1)$,
\be\label{takis30} 
\ds \int_{|x|\leq R} \rho(x,t)  dx  \geq \dfrac{1}{2}. 
\ee

Indeed,   \eqref{ph140.0} yields that 
\[R^2 \; \underset{t\geq h}\sup \ds \int_{\{|x|\geq R\}} \rho(x,t) dx \lesssim_h 1,\] 
and, thus, since $\ds \int_\R \rho(x,t)dx=1$, \eqref{takis30} holds for sufficiently large $R$. 

\end{proof}



The second preliminary estimate is a uniform upper bound on  
\be\label{brod3}
 \langle u_{\ep,\eta}\rangle_R (t)=\dfrac{1}{2R} \int_{-R}^R u_{\ep,\eta}(x,t) dx
 \ee
 by a quantity that depends on $\ds \int_{-R}^R (\pt_x u_{\ep,\eta})^2(x,t) dx dt$.
\vs

Below to simplify the formulae we write 
\be\label{brod5}
\phi_{\ep,\eta}(t)=\int_R \rho_{\ep,\eta}^2(x,t)dx,
\ee
and note that 
\[\ds\int_0^1 \phi_{\ep,\eta}(t) dt \lesssim 1.\]

\begin{lem}\label{brod2}
 Uniformly in $\ep$ and $\eta$ satisfying  \eqref{uc0} and all $R\geq R_0$ and $h\in (0,1)$, there exists $C=C_h$ such that 
 \be\label{brod4}
 \langle u_{\ep,\eta}\rangle_R (t) \leq C \Big[1+ R \phi_{\ep,\eta}(t)^{1/2} \Big(\ds\int_0^R (\pt_x u_{\ep,\eta})^2(x,t) dx \Big)^{1/2}\Big].
\ee
\end{lem}

\begin{proof}

To simplify the notation, in what follows  we write $\rho$ and $u$ instead of  $\rho_{\ep,\eta}$ and $u_{\ep,\eta}$ and we denote by  $C$ constants that may change from line to line but  are uniform in $\ep$ and $\eta$ and may depend on $h$.  Finally, all estimates below will be uniform in $\ep$ and $\eta$. 
\vs

A preliminary remark is 
 is that 
 that, since $\rho \geq 0$, 
$u$ is a supersolution to  $$\ds \pt_t w- \ep \pt_{x}^2w  +\dfrac{1}{2} (\pt_x w)^2 =0\ \text{in} \ \rt \ \  \text{and} \ \  w(x,0)=\frac{x^2}{2\eta},$$ 
whose solution is 
\be\label{takis160}
w(x,t)= \dfrac{x^2}{2(\eta+t)} +\ep\log \big(\dfrac{\eta +t}{\eta}\big). 
\ee
Using that  $\log\big((\eta+t)/\eta\big) \geq 0$, we find  that 
\be\label{takis23.1}
 u \geq w \geq \dfrac{x^2}{2(\eta+t)} \ \ \text{in} \ \ \R\times [0,1],
 \ee
and, hence,
\[\langle u\rangle_R (t) \geq \dfrac{R^2}{6(\eta+t)}.\]



\vs
%
In what follows we observe that, since both $u$ and $\rho$ are even functions, we may  integrate in space over $(0,R)$ rather than $(-R,R)$. In particular, we note that 
\[\langle u\rangle_R (t) = \frac{1}{R} \ds \int_0^R u(x,t) dx \ \ \text{and} \ \   \ds\int_{|x|\leq R} (\pt_x u)^2(x,t) dx dt=2 \ds\int_0^R (\pt_x u)^2(x,t) dx dt.\]


%
\vs

Using \eqref{takis30}, Poincare's inequality and the  bound on $\ds \int_0^1  \int_\R \rho(x,t) u(x,t) dx dt$ in the previous proposition, we find that, for all $R\geq R_0$,
\begin{equation}\label{takis30.1}
\bs
 \langle u \rangle_\R (t)  \leq  & \ds  2 \int_{|x|\leq R} \rho(x,t)  \langle u \rangle (t) dx  \leq   \ds 2\big( \int_{|x|\leq R} \rho(x,t) |u- \langle u\rangle(t)| dx + 
 \ds  \int_{|x|\leq R} \rho(x,t) u(x,t) dx \big) \\[1.5mm]
 & \leq C\Big[1 + R \Big( \int_{|x|\leq R} \rho^2(x,t) dx\;\Big)^{1/2} \Big(\int_{|x|\leq R} |u(x,t)-\langle u\rangle (t)|^2 dx\; \Big)^{1/2}\Big]\\[1.5mm]
&\leq C \Big[1+ R \phi(t)^{1/2} \Big(\ds\int_0^R (\pt_x u)^2(x,t) dx \Big)^{1/2}\Big ].
 \end{split}
 \end{equation}

 \end{proof}
%
%



\vs

We are now ready to present the proof of  the ``coming down from infinity argument''.

\begin{proof}[The proof of Proposition~\ref{prop12}]

Let 
\[F(R,h) = \ds \int_h^1 \ds \int_0^R (\pt_x u)^2(x,t) dx dt.\]
\vs

Integrating the first equation in \eqref{takis1} over $[0,R]\times[h,1]$, using that $\ep\in(0,1), \ R_0\geq 1$, $u\geq0$, $\ds \int_\R \rho(x,t)dx=1$ and \eqref{takis30.1}, after some elementary rearrangement of terms, we find 
\begin{equation*}
\bs
F(R,h) &\leq  2\Big ( 1 + \ds \int_0^R u(x,h) dx + \ep \ds \int_h^1 \pt_x u (R,s) ds\Big) \\[1.5mm]
&\leq  \Big( 1 +  R \langle u\rangle_R (h)   + \ep \ds \int_h^1 \pt_x u (R,s) ds\Big) \\[1.5mm]
& \leq C  \Big[1 + \Big(\ds\int_h^1 (\pt_x u)^2(R,s) ds \Big)^{1/2} + R + R^2 \phi(h)^{1/2}  \Big(\ds\int_{|x|\leq R} (\pt_x u)^2(x,t) dx \Big)^{1/2}\Big ]\\
& \leq C \Big[  1 +  \big(\dfrac{\pt}{\pt R} F(R,h)\big)^{1/2}  +  \phi(h)^{1/2}R^2 \Big( -\dfrac{\pt}{\pt h }F(R,h)\Big)^{1/2} \Big],\\[1.5mm]
\end{split}
\end{equation*}
and, finally,

\be\label{takis33}
F(R,h)^2 \leq  C \Big[ 1+   \dfrac{\pt}{\pt R} F(R,h) - R^4 \phi(h) \dfrac{\pt}{\pt h}F(R,h)\Big]. 
\ee
\vs


Next we fix a pair $(\bar R,\bar h)$ with $R\geq R_0$ and $h\in (0,1/2)$ and write $\bar F=F(\bar R, \bar h)$.
\vs

We consider the characteristic odes
\be\label{takis34}
\begin{cases}
\dot R=1   \ \ \text{in}  \ \ (0,\infty)  \ \ \text{and } \ \ R(0)=\bar R,\\[1.2mm]
\dot h=- R^4\phi(h) \ \ \text{in}  \ \ (0,\infty)  \ \ \text{and}  \ \   h(0)=\bar h. 
\end{cases}
\ee
Let $\hat t>0$ be the time that the decreasing curve $t\to h(t)$ hits $0$, that is, 
\be\label{brod8}
h(\hat t)=0.
\ee

It follows that 

\be\label{brod7}
F^2(R(t),h(t))\leq C^2\big[1 + \dfrac{d}{dt} F(R(t),h(t))\big] \ \ \text{in} \ \ (0,\hat t).
\ee

Since  $h\to F(R,h)$ is decreasing in $h$ and $F(R,0)=0$, it follows that, for all $R\in [\bar R, R(\hat t)]$,
there exists $h(R) \in (\bar h,R)$ such that 
\be\label{brod8}
\text{either} \ \ F^2(h,R) \leq 2C^2 \ \ \text{in} \ \ [h(R),1] \ \ \text{or} \ \ 2C^2< F^2 \ \ \text{in} \ \ (0,h(R)).
\ee

If $\bar F\leq 2C^2$ we have the desired upper bound, so for the rest of the proof was assume that the other alternative in \eqref{brod8} holds. Then \eqref{brod7} yields 

\be\label{brod9}
F^2(R(t),h(t))\leq 2C^2 \dfrac{d}{dt} F(R(t),h(t)) \ \ \text{in} \ \ (0,\hat t).
\ee

It  follows that 
\be\label{brod10}
\dfrac{\hat t}{2C^2} \leq -\dfrac{1}{F(R(\hat t),h(\hat t))} + \dfrac{1}{F(R(0),h(0))}\leq \dfrac{1}{\bar F},
\ee
and, thus, 
\be\label{brod11}
\bar F\leq \dfrac{2C^2}{\hat t}.
\ee

To conclude,  we obtain a lower bound for $\hat t$ as follows. 
\vs

Rewriting the second equation in \eqref{takis34} as 
\[\dfrac{\dot h}{\phi(h)}=-R^4,\]
and using the explicit form of $R(t)=\bar R + t$,  we find, after  integrating over $(0,\hat t)$, 
\be\label{brod10}
\ds \int_0^{\bar h} \frac{dh}{\phi(h)}=\frac{(\bar R +\hat t)^5 -\bar R^5}{5}
\ee 
Jensen's inequality implies that 
\be\label{brod11}
\ds \int_0^{\bar h} \frac{dh}{\phi(h)}\geq h^2 \big (\int_0^{\bar h} \phi(h) dh \big )^{-1}.
\ee
Since 
\[\ds \int_0^{\bar h} \phi(h) dh \leq \int_0^{1} \phi(h) dh=\int_0^1\int_\R \rho^2(x,s) dx ds \lesssim 1\]
and
\[\frac{(\bar R +\hat t)^5 -\bar R^5}{5}\lesssim_{\bar R} \hat t,\]
we find that 
\[(\bar h)^2\lesssim \hat t. \]
It follows from \eqref{brod11} that 
\[\ds F(\bar R, \bar h) \lesssim_{\bar R} \frac{1}{\bar h^2},\]
and, thus, the claim.

\vs

The bound on
$u$ in $L^\infty((h,1); L^2_{\text{loc}}(\R)),$ which follows from \eqref{takis30.1}. 
\vs


For the last bound, we fix $R'>R\geq R_0$ and note that,  since $\ds \int_{h/2}^1 \ds \int_0^{R'} (\pt_x u)^2(x,s) dx ds$ is bounded independently of $\ep$ and $\eta$, there exists $R''\in (R,R')$ such that 
\be\label{takis65}
\ds \int_{h/2}^1  (\pt_x u)^2(R'',s) ds \ \ \text{ is bounded independently of $\ep$ and $\eta$}.
\ee
\vs

We also know from the first estimate   and the fact that $\int_0^1\int_\R \rho^2 dx ds\lesssim 1$  that 
\[ \ds \int_{h/2}^\theta \ds \int_0^{R''} \big[(\pt_x u)^2(x,s) +  \rho^2(x,s)\big] dx ds \ \ \text{ is  bounded uniformly in $\ep$ and $\eta$}.\]
\vs

It follows that there exists some 
$\hat h\in [h/2, h]$ such that 
\be\label{takis36}
\ds \int_0^{R''} \rho^2(x,\hat h) dx  \ \ \text{and} \ \ \ds\int_0^{R''} (\pt_x u)^2(x,\hat h) dx \ \ \text{are bounded uniformly in $\ep$ and $\eta$},
\ee
which together with  \eqref{takis30.1} imply that 
\be\label{takis37}
\ds \int_0^{R{''}}u(x,h) dx \ \ \text{is bounded independently of $\ep$ and $\eta$}.
\ee

\vs

Integrating the first equation in \eqref{takis1} over $[0,R''] \times [h,1]$ we find 
\be\label{takis66}
\ds \int_0^{R''} u(x,t) dx \leq 1 + \ds \int_0^{R''} u(x,h) dx +\ep\Big( \ds \int_h^1 (\pt_x u)^2(R'',s)ds \Big)^{1/2}. 
\ee
\vs
In view of the above observation and \eqref{takis65} we know that  the right hand side of \eqref{takis66} is bounded independently of $\ep$ and $\eta$, so 
the  same is true for $\ds \int_0^{R{''}} u(x,t) dx$, and, hence, $\ds \int_0^{R} u(x,t) dx$.

\end{proof}

Using the system satisfied by $(\pt_x u_{\ep,\eta}, \pt_x \rho_{\ep,\eta})$  we prove next  a  new estimate for $\pt_x \rho_{\ep,\eta}$ and  $\sqrt{\rho_{\ep,\eta}}\pt_x u_{\ep,\eta}$ in $L^2_{t, \text{loc}}(L^2_x(\R))$, which is uniform in $\ep$ and $\eta$ satisfying \eqref{uc0}. The reasons for the $L^2_{t,\text{loc}}$ are  the  boundary conditions of $\rho_{\ep,\eta}$ at  $t=1$ and $u_{\ep,\eta}$ at $t=0$. %

\begin{prop}\label{prop125}
 Uniformly in $\ep, \eta$ satisfying \eqref{uc0} and, for any $h \in (0,1/2)$,
 \be\label{ma31}
 \int_h^{1-h}\int_\R \Big[(\pt_x \rho_{\ep, \eta})^2(x,t) + \big(\rho_{\ep,\eta} \big(\pt_x^2u_{\ep,
 \eta}\big)^2\big)(x,t)\Big] dxdt \lesssim_h 1 \ \ \text{and} \ \ 
  \|\rho_{\ep,\eta}\|^2_{L^2(h, 1-h ;H^1(\R))} \lesssim_h 1. 
 \ee
%
%
%
%
\end{prop}
\begin{proof}
As in the previous proof,  we omit the dependence of $u_{\ep,
\eta}$ and $\rho_{\ep,\eta}$ on $\ep$ and $\eta$. We also note that all the estimates below are uniform in $\ep$ and $\eta$ satisfying  \eqref{uc0}.
\vs

A simple computation, which can be justified by considering  approximate problems with regularized $\rho(\cdot, 1)$ and then passing to the limit,  leads, for each $t\in (0,1)$,  to the identity 
\be\label{takis40.11}
\dfrac{d}{dt}\ds\int_\R \pt_x \rho(x,t) \pt_xu(x,t) dx= \ds\int_\R \big[(\pt_x\rho)^2 (x,t) + \rho(x,t) (\pt_x^2 u)^2(x,t) \big]dx.
\ee
\vs 

Next we localize the estimate in time staying away from $0$ and $1$. For this, we consider a smooth cut-off function $\chi$ supported in $(h, 1-h)$ and  multiply \eqref{takis40.11} by $\chi^2$. 
\vs
Integrating in time we get

\[ 
\ds \int_0^1\ds  \chi^2(t) \dfrac{d}{dt}\ds\int_\R \pt_x\rho(x,t) \pt_xu(x,t) dx dt= \ds \int_0^1 \ds\int_\R \chi^2(t) \big[(\pt_x\rho)^2 (x,t) + \rho(x,t) (\pt_x^2 u)^2(x,t) \big] dx dt, \]
\vs
and, hence,

\begin{equation*}
\bs
\ds \int_0^1 \ds\int_\R \chi^2(t) \big[(\pt_x\rho)^2 (x,t) &+ \rho(x,t) (\pt_x^2 u)^2(x,t) \big]dx dt =-2 \ds \int_0^1 \ds\int_\R \chi(t) \chi'(t)\pt_x\rho(x,t) \pt_xu(x,t) dx dt \\[1.5mm]
&= 2 \ds \int_0^1 \ds\int_\R \chi(t) \chi'(t)\rho(x,t) \pt_x^2u (x,t) dx dt \\[1.5mm]
&\leq \dfrac{1}{2} \ds \int_0^1 \ds\int_\R (\chi(t))^2 \rho(x,t) ( \pt_x^2 u)^2 (x,t) dxdt + 2  \ds \int_0^1 \ds\int_\R (\chi')^2(t) \rho (x,t) dxdt,
\end{split} 
\end{equation*}
\vs
which leads to the claim in view of the fact that, for some $C$ which depends on the support of $\chi$, 
\[ \ds \int_0^1 \ds\int_\R (\chi')^2(t) \rho(x,t) dxdt \leq C \ds \int_0^1 \ds\int_\R \rho(x,t) dxdt =C. \]
Note that the restriction on the relative size of $\ep$ and $\eta$ played role in this estimate. It is needed, however, for the second estimate in \eqref{takis40.11} which requires the previously shown bound on $\rho$ in $L^2_{t,\text{loc}}(L^2_x)$.

\end{proof}

We use  the estimates from the previous proposition to obtain, a uniform in $\ep$ and $\eta$ satisfying \eqref{uc0},  control of 
the  $C_x^{0,\gamma}-$norm of $\rho_{\ep,\eta}$ in $L^2(h,1)$ for all $h \in (0,1)$. We note that this a surprising estimate in view of the fact that  $\rho_{\ep,\eta}(\cdot,1)=\text{\boldmath $\delta$}.$

%
\vs

\begin{prop}\label{prop15}  Uniformly in $\ep$ and $\eta$ satisfying \eqref{uc0}, 
 for each $h\in (0,1)$ and  for all $\gamma\in (0,1/2)$, 
\be\label{takis70}
\ds \int_{h}^{1} [\rho_{\ep,\eta}(\cdot,t)]_{C^{0,\gamma}} dt \lesssim_h 1.
\ee
\end{prop}
\begin{proof}
As in the previous proofs we drop the dependence on $\ep$ and $\eta$.
\vs

We present the argument for $h=1/2$ and leave it up to the reader to adjust the argument for $h\in (0,1).$
\vs

\vs


The starting point is the identity \eqref{takis40.11}    
which we  rewrite as
\be\label{takis72}
\dfrac{d}{dt} \ds\int_\R [-\rho(x,t) \pt_x^2 u(x,t)]dx=\ds\int_\R \big[(\pt_x\rho)^2 (x,t) + \rho(x,t) (\pt_x^2 u)^2(x,t) \big]dx.
\ee
\vs

Then

\[-\dfrac{d}{dt} \ds\int_\R [- \rho(x,t) \pt_x^2 u(x,t)]dx \geq \dfrac{1}{2} \big(-\ds\int_\R \rho(x,t) \pt_x^2 u(x,t) dx\big)^2,\]
\vs
and, hence, for all $t\in (1/2,1)$,
\vs
\be\label{takis73}
\ds\int_\R [-\rho(x,t) \pt_x^2 u(x,t)]dx \leq \dfrac{1}{2(1-t)}.
\ee
\vs


Next, we observe that, for every $t \in (0,1)$, 
\be\label{ma35}
\bs \int_\R  \pt_x\rho (x, t ) \pt_x u (x, t) dx&=-\int_\R \rho(x,t) \pt_x^2 u(x,t)dx\\
& \leq \Big( \int_\R \rho(x,t) dx\Big)^{1/2} \Big(\int_\R \rho(x,t) (\pt^2_x u)^2(x,t)\Big)^{1/2}\\ 
& =\Big(\int_\R \rho(x,t) (\pt^2_x u)^2(x,t)\Big)^{1/2}.
\end{split}
\ee

It also follows from Proposition~\ref{prop125} that, for $\ep$ and $\eta$ satisfying \eqref{uc0}, there exists $\hat t=t_{\ep,\eta}\in (1/4, 1/2)$ such that 
\[\int_\R \big[ (\pt_x\rho)^2(x,\hat t ) + \rho(x,\hat t) (\pt_x u )^2(x,\hat t )\big]dx \lesssim 1. \]

Then  \eqref{ma35} yields
\be\label{takis74}
\int_\R  \pt_x\rho (x,\hat t ) \pt_x u (x,\hat t) dx \lesssim 1.
\ee
\vs

Integrating  \eqref{takis72}  over $(\hat t, t)$  we find
\vskip.05in 
\[\bs
\ds\int_{1/2}^t\ds \int_\R (\pt_x\rho)^2(x,t) dx dt 
&\leq \ds\int_{\hat t }^t\ds \int_\R (\pt_x\rho)^2(x,t) dx dt \leq    \ds \int_\R \big[-\rho (x,\hat t) \pt^2_x u(x,t) dx - \int_\R [-\rho (x,\hat t) \pt^2_x u(x,\hat t)]  dx, \\
& = \ds \int_\R \big[-\rho (x,\hat t) \pt^2_x u(x,t) dx + \int_\R [\rho (x,\hat t) \pt^2_x u(x,\hat t)]  dx, 
\end{split}
  \]
\vs
and, in view of \eqref{takis73} and \eqref{takis74}, 

%

\be\label{takis75}
\ds\int_{1/2}^t\ds \int_\R (\pt_x\rho)^2(x,t) dx dt \leq \dfrac{1}{2(1-t)}.
\ee
\vs
Next, we recall that, for every $\gamma \in (0,1/2)$, there exists $l \in (0,1)$ such that 
\[ [\rho(\cdot,t)]_{C^{0,\gamma}}^{2+l} \lesssim_\gamma \ds\int_\R (\pt_x \rho)^2(x,t) dx,\]
which together with \eqref{takis75} imply that 

\[\ds \int_{1/2}^t [\rho(\cdot,s)|_{C^{0,\alpha}}^{2+\beta} ds \lesssim \dfrac{1}{2(1-t)}.\]
\vs

Then Lemma~\ref{lem1}, which is stated and proved below, yields \eqref{takis70}.

\end{proof}

We state and prove next the auxiliary fact we used at the end of the proof above. For simplicity we change $1-t$ to $t$.

\begin{lem}\label{lem1}
Let $a:(0,1)\to [0,\oo)$ and $m>2$ be such that, for $h, h'\in (0,1/2)$ with $h<h'$, $\int_h^{h'} a^m(s) ds\leq C/(h'-h)$. Then $\int_0^{1/2} a(s) ds \leq C.$ 
\end{lem}
\begin{proof}
Observe that 
\[\begin{split}
&\int_0^{1/2} a(s)ds=\sum_{n\geq 0} \ds \int_{1/2^{n+1}}^{1/2^n} a(s)ds \leq 
 \sum_{n\geq 0} \Big[\Big(\ds \int_{1/2^{n+1}}^{1/2^n} a^m(s) ds \Big)^{1/m} \Big( 2^{-(n+1)}\Big)^{(m-1)/m}\Big]\\[1.2mm]
&\leq C\sum_{n\geq 0} {2^{{(n+1)/m}}} 2^{-(n+1)((m-1)/m)}= C\sum_{n\geq 0} 2^{-(n+1)((m-2)/m)},
\end{split}\]
and $\sum_{n\geq 0} 2^{-(n+1)((m-2)/m)}<\oo$ if $\frac{m-2}{m}>0$, that is, $m>2$.

\end{proof}

%
%
\vs

The next proposition provides uniform in $\ep$ and $\eta$ satisfying \eqref{uc0} and $t\in (h,1)$  bounds for the $C^{0,\gamma}-$norm  of  $u_{\ep, \eta}$  and, thus,  improves the   available regularity up to and including $t=1$. \vs 

\begin{prop}\label{prop100}
Uniformly in $\ep$ and $\eta$ satisfying  \eqref{uc0}  and  for every $h \in (0,1)$ and $\gamma \in (0,1)$, 
\be\label{takis71}
\underset{t \in (h,1]} \sup \|u_{\ep, \eta}(\cdot, t) \|_{\oo,\text{loc}} \lesssim_h 1 \ \ \text{and} \ \ 
\underset{t\in (h,1]}\sup \|u_{\ep, \eta}(\cdot, t)\|_{C^{0,\gamma}_{\text{loc}}} \lesssim_{\gamma, h} 1.
\ee
\end{prop}
\begin{proof}  As in the previous proofs we omit the explicit dependence on $\ep$ and $\eta$.
\vs
In view of Proposition~\ref{prop15}, 
the claim follows from the theory of viscous parabolic Hamilton-Jacobi equation once we establish a, uniform in $\ep$ and $\eta$ satisfying \eqref{uc0},  bound on
$\underset{t \in (h,1]} \sup \|u_{\ep, \eta}(\cdot, t) \|_{\oo,\text{loc}}$. 
\vs

Proposition~\ref{prop12} yields that, for every $h\in (0,1)$ and $R>0$, 
%
\be\label{takis50.2}
\underset{t\in (\theta,1)} \sup \ds \int_{-R}^R u(x,t) dx + \ds \int_h^{1} \ds \int_{-R}^R (\pt_x u)^2(x,t)dxdt  \lesssim_{h}  R^2. 
\ee
We also remark that, if 
\[\ds \int_h^{1} \ds \int_{-R}^R (\pt_x u)^2(x,t)dxdt  \lesssim_{h}  R^2,\]
then, for each $m>2$, 
\[\ds \int_h^1 \ds \int_\R \frac{ u^2(x,t) + (\pt_x u)^2(x,t)}{(1+x^2)^{m/2}} \lesssim_{h,m} 1.\]

\vs
We fix $R>0$ and  rewrite \eqref{takis50.2} and the above estimate as
\be\label{takis50.222}
\underset{t\in (h,1)} \sup \ds \int_{-R}^R u(x,t) dx + \ds \int_h^{1} \ds \int_\R (\pt_x u)^2(x,t)dxdt + \ds\int_h^1 \ds \int_\R \frac{(\pt_x u)^2(x,t)}{(1+x^2)^{m/2}} dx dt\lesssim_{h,m,R} 1. 
\ee
\vs
It follows that there exist  $\hat t=t_{\ep,\eta}^{h} \in (h/2, h)$ and $\hat x=x_{\ep,\eta}^{h}\in (-R,R)$ such that 
\be\label{takis50.11}
(i)~\ds  \int_{-R}^R u(x,\hat t) dx + \ds  \int_{-R}^R (\pt_x u)^2(x,\hat t) dx +\ds \int_\R \frac{(\pt_x u)^2(x,t)}{(1+x^2)^{m/2}} dx \lesssim_{h,R} 1  \ \ \text{and} \ \ 
(ii)~u(\hat x, \hat t) \lesssim_h 1.
\ee
\vs

Finally, since the $\rho_{\ep,\eta}$'s are  bounded in $L^1((\hat t, 1); L^\infty(\R))$, we get, for $(x,t)\in \R \times (h,1]$, 
\[u(x,t)\leq v(x,t) + \ds \int_{t_0}^{t} \|\rho(\cdot, s)\|_\infty ds,\]
where $v$ solves
\[\pt_t v-\ep \pt^2_x v + \dfrac{1}{2}( \pt_x v)^2=0 \ \ \text{in} \ \ R\times (\bar t,1] \ \ \text{and} \ \  v(\cdot, \bar t)=u(\cdot,  \bar t).\]
\vs

Then, for each progressively measurable $\alpha:(\hat t,1)\to \R$, we have 
\be\label{takis52.1}
v(x,t)\leq\E \Big[\ds \int_{\bar t}^t \dfrac{1}{2}|\alpha_s|^2 + u (X_t, \hat t)\Big],
\ee
where $dX_s=\alpha_s ds + \sqrt{2 \ep} dW_s$ and $X_{\hat t}=x$.
\vs

We choose next  the constant in $s$ control $\alpha_s=\dfrac{x_0-x}{t-\hat t}$. Then, for $\tau=t-\hat t$, 
\[ X_t=x  + \dfrac{\hat x-x}{t-\hat t}(t-\hat t) +  \sqrt{2\ep} dW_\tau=x_0 +  \sqrt{2\ep} dW_\tau ,\]
and \eqref{takis52.1} yields that
\be\label{takis53.12}
v(x,t) \leq \dfrac{1}{2}\dfrac{(x-\hat x)^2}{t-\hat t} + \E[u(\hat x + \sqrt{2\ep}W_\tau, \hat t)].
\ee
\vs 
Moreover, 
\begin{equation}\label{takis54}
\bs
\E[u (\hat x + \sqrt{2\ep}(W_\tau, \hat t)]&=\dfrac{1}{\sqrt{2\ep\tau}}\ds\int_\R u(\hat x +y, \hat t) e^{-\frac{y^2}{2\ep\tau}}dy\\[1.5mm]
& \leq u (\hat x, \hat t) + \dfrac{1}{\sqrt{2\ep\tau}}\ds\int_{|y|\leq 1}|u (\hat x +y, \hat t)-u (\hat x, \hat t)| e^{-\frac{y^2}{2\ep\tau}}dy\\
& + \dfrac{1}{\sqrt{2\ep\tau}}\ds\int_{|y|\geq 1} u (\hat x +y, \hat t) e^{-\frac{y^2}{2\ep\tau}} dy.
\end{split}
\end{equation}
\vs

Note that  \eqref{takis50.11}(i) yields that 
\[\sup_{|y|\leq 1} |u(\hat x + y, \hat t)-u (\hat x,\hat t)| \lesssim_{h, R} 1, \]
and
\[\dfrac{1}{\sqrt{2\ep\tau}}\ds\int_{|y|\geq 1} u(\hat x +y, \hat t)e^{-\frac{y^2}{2\ep\tau}}dy 
\leq \int_{|y|\geq 1}\dfrac{u (\hat x +y, \hat t)}{(1+|y|^2)^{m/2}} \sup_{|z|\geq 1} [{|z|^m} 
\frac{e^{\frac{-|z|^2}{\sqrt{2\ep\tau}}}}{\sqrt{2\ep\tau}}\;dy\lesssim_{\theta} 1.\]
\vs

Combining \eqref{takis50.11}, \eqref{takis52.1},  \eqref{takis53.12} and the last two estimates we find 

%
\be\label{takis55}
u(x,t)-\dfrac{(x-\hat x )^2}{2(t-\hat t)} \lesssim_h 1.
\ee

\end{proof}

%
%
%
%

The  estimate presented in the next proposition is about the behavior of 
$u_{\ep,\eta}$  near $t=0$. The claim is that, for $\ep$ and $\eta$ satisfying \eqref{uc0}, the $u_{\ep,\eta}$'s are bounded in the sense of distributions from above by $x^2/2t$, a fact that explains that, as $\ep,\eta \to 0$, the effect of the diffusion is negligible. 
\vs

\begin{prop}\label{prop20.0}
Uniformly in $\ep$ and $\eta$ satisfying  \eqref{uc0} and  for any  nonnegative compactly supported and smooth $\chi:\R\to [0,\infty)$,  there exists, an  independent of $\ep$ and $\eta$, map $e:\R\to \R$ such that $\underset{t\to 0}\lim\; e(t)=0$ and, for $t\in (0,1)$ and the $\kappa$ in \eqref{uc0}, 
\be\label{takis80}
\underset{ \ep,\eta, \ep(-\log \eta)^\kappa \to 0}\limsup \ds \int_\R u_{\ep, \eta}(x,t) \dfrac{1}{t^{2/3}}\chi(\dfrac{x}{t^{2/3}})dx\\ 
\leq \ds \int_\R \dfrac{x^2}{2t} \dfrac{1}{t^{2/3}}\chi(\dfrac{x}{t^{2/3}}) dx + e(t).
\ee 
\end{prop}

\begin{proof}
Like in previous proofs,  we write $u$ and $\rho$ in place of $u_{\ep,\eta}$ and $\rho_{\ep,\eta}$. Moreover, we note that, if $\ep$ and $\eta$ satisfy \eqref{uc0}, then $\lim \big(\ep \log \eta\big)=0.$
\vs


Let $v:\R\times [0,1]\to \R$ be given by 
\[u(x,t)=\dfrac{x^2}{2(t+\eta)} +  \ep (\log(t+\eta) -\log \eta)) +v(x,t). \]
 It is immediate that  
 \be\label{takis190}
 \pt_t v -\ep \pt_x^2 v +\dfrac{1}{2} (\pt_x v)^2 + \dfrac{x}{t+\eta} \pt_x v=\rho  \ \text{and} \ v\geq 0  \ \text{in} \ \ \rt \  \text{and} \  v(\cdot,0)=0.
 \ee
We also note that, as $\ep,\eta, \ep |\log \eta|  \to 0$,  $\underset{t\to 0} \lim \ep (\log(t+\eta) -\log \eta)=0$ uniformly in $t\in (0,1)$. 

\vs

Given $\chi$ as in the statement, set $\chi_\eta(x,t)=\dfrac{1}{(t+\eta)^{2/3}} \chi (\dfrac{x}{(t+\eta)^{2/3}})$.
\vs

Then 
\be\label{takis81}
\begin{split}
\dfrac{d}{dt}\ds \int_\R v\chi_\eta dx- &\ds \int_\R v \pt_t\chi_\eta dx + \dfrac{1}{2}\ds \int_\R (\pt_x v)^2\chi_\eta dx  + \ds \int_\R \dfrac{x}{t+\eta} \pt_x v \chi_\eta dx\\[1.2mm]
&  -  \ep \ds\int_\R  \pt_x^2 v \chi_\eta dx  =\ds \int_\R \rho \chi_\eta \lesssim_{\chi} \dfrac{1}{(t+\eta)^{2/3}},
\end{split}
\ee
the last inequality following from the facts that $\chi$ is bounded and $\int_\R\rho(x,t) dx=1$.
\vs

Next we analyze  the terms in the left hand side of \eqref{takis81} beginning with the last.
\vs

Integrating by parts we find
\[
\begin{split}
|\ds\int_\R  \pt_x^2 v \chi_\eta dx| &\leq \frac{1}{(t+\eta)^2} \ds \int_\R |v(x,t)| |\pt_x^2 \chi( \frac{x}{(t+\eta)^{2/3}})| dx \\
& \leq \frac{1}{(t+\eta)^{4/3}} \|v(\cdot,t)\|_\oo \ds \int_\R |\pt_x^2 \chi(x)| dx \lesssim_{\int_\R |\chi''(x)|dx} \frac{1}{(t+\eta)^{4/3}} \|v(\cdot,t)\|_\oo.
\end{split}
\]
%
\vs
Since $\int_\R \rho dx=1$, it also follows from \eqref{takis190} that
\[\dfrac{d}{dt} \ds \int_\R v dx - \dfrac{1}{t+\eta} \ds  \int_\R vdx + \ds \int_\R (\pt_x v)^2 =1,\]
and, thus,  
\[\dfrac{d}{dt} \Big( \dfrac{1}{t+\eta} \ds \int_\R vdx \Big)+ \dfrac{1}{t+\eta} \ds \int_\R  (\pt_x v)^2 dx =\dfrac{1}{t+\eta},
\]
which then yields
\be\label{takis191}
 (i)~\ds \int_\R v(x,t) dx \leq (t+\eta) \text{log} \dfrac{t+\eta}{\eta} \ \  \text{and} \ \  (ii)~
\ds \int_0^t\dfrac{1}{s+\eta}\ds \int_\R (\pt_x v)^2(x,s) dx ds\leq  \text{log} \dfrac{t+\eta}{\eta}.
\ee
\vs
Using next the   imbedding $\|w\| \lesssim \|w\|^{1/3}\|\pt_x w\|^{2/3}$, we find from \eqref{takis191}  that, for each $t\in (0,1)$, 
\[
\| v(\cdot,t)\|_\oo \leq c \Big((t+\eta) \text{log} \dfrac{t+\eta}{\eta}\Big)^{1/3} \|\pt_x v(\cdot,t) \|_\oo^{2/3}, 
\]
and, therefore, in view of \eqref{takis191}(ii),
\[\ds\int_0^t \dfrac{1}{(s+\eta)^2}\dfrac{1}{\text{log} \dfrac{s+\eta}{\eta}} \|v(\cdot,s)\|_\oo^3 ds \lesssim \; \text{log} \dfrac{t+\eta}{\eta}. \]
Then,
\[
\bs
\ep \ds\int_0^t \dfrac{1}{(s+\eta)^{4/3}}\|v(\cdot,s)\|_\oo ds& =\ep \ds\int_0^t \big[\dfrac{1}{(s+\eta)^{2/3}}\dfrac{1}{(\text{log} \dfrac{s+\eta}{\eta})^{1/3}} \|v(\cdot,s)\|_\oo\big] \big[\dfrac{1}{(s+\eta)^{2/3}} (\text{log} \dfrac{s+\eta}{\eta})^{1/3}\big]ds 
\\[1.2mm]
&\lesssim \ep \big(\text{log}\dfrac{t+\eta}{\eta} \big)^{1/3}
\Big( \ds \int_0^t \dfrac{1}{s+\eta} \big( \text{log} \dfrac{s+\eta}{\eta}\big)^{1/2}ds)\Big)^{2/3}
\lesssim  \ep \text{log} \frac{t+\eta}{\eta}, 
\end{split}
\]
and, thus, for the $\kappa$ in \eqref{uc0}, 
\be\label{takis193}
\underset{\ep, \eta, \ep(-log \eta)^\kappa \to 0}\limsup \Big[ \ep \ds \int_0^t \int_\R \pt^2_x v \chi_\eta dx ds|\Big] \leq 0.
\ee

\vs

For  the second term in \eqref{takis81} we find  

\[
\begin{split}
\ds \int_\R v \pt_t\chi_\eta &= -\dfrac{2}{3}\dfrac{1}{(t+\eta)}\ds\int_\R v\chi_\eta dx - \dfrac{2}{3} \dfrac{1}{(t+\eta)^{2/3} }\ds\int_\R v \chi'(\dfrac{x}{(t+\eta)^{2/3}})\dfrac{x}{(t+\eta)^{5/3}} dx\\[1.2mm]
&= -\dfrac{2}{3}\dfrac{1}{(t+\eta)}\Big[ \ds\int_\R v\chi_\eta dx + \ds\int_\R v \pt_x \chi_\eta x dx \Big]\\[1.2mm]
& = -\dfrac{2}{3}\dfrac{1}{(t+\eta)} \Big[   \ds\int_\R v\chi_\eta dx - \ds\int_\R v\chi_\eta dx -  \ds\int_\R  \pt_x  v \chi_\eta x dx \Big]\\[1.2mm]
& = \dfrac{2}{3}\dfrac{1}{(t+\eta)}  \ds\int_\R  \pt_x  v \chi_\eta x dx\\[1.2mm]
& \leq \dfrac{2}{3} \Big[ \dfrac{1}{8} \ds\int_\R ( \pt_x  v)^2 \chi_\eta dx + 2 \ds\int_\R \chi_\eta(\frac{x}{(t+\eta)^{2/3} })\dfrac{x^2}{(t+\eta)^2} dx \Big]\\[1.2mm]
& \leq \dfrac{2}{3} \Big[ \dfrac{1}{8} \ds\int_\R ( \pt_x  v)^2 \chi_\eta dx +2 \dfrac{1}{(t+\eta)^{2/3}} \ds\int_\R \chi_\eta \big(\dfrac{x}{(t+\eta)^{2/3}}\big) \dfrac{x^2}{(t+\eta)^{4/3}} dx \Big]\\[1.2mm]
& \lesssim \Big[  \ds\int_\R ( \pt_x  v)^2 \chi_\eta dx + \dfrac{1}{(t+\eta)^{2/3}}\Big].
\end{split}
\]        
\vs
%


Finally, since 
\[
| \ds \int_\R \dfrac{x}{t+\eta} \pt_x v \chi_\eta dx | \leq \dfrac{1}{8} \ds \int_\R (\pt_x v)^2\chi_\eta dx + 2 \ds \int_\R \dfrac{x^2}{(t+\eta)^2} \chi_\eta dx \lesssim \ds \int_\R (\pt_x v)^2\chi_\eta dx  + \dfrac{1}{(t+\eta)^{2/3}},\]
%

we find 
\[
\ds \int_\R (v\chi_\eta)(x,t) dx \lesssim \Big[ \ds \int_0^t \dfrac{1}{(s+\eta)^{2/3}} ds + \ep|\ds \int_0^t \int_\R \pt^2_x v \chi_\eta dx. ds|  \Big],
\]
It follows that 
\be\label{takis83}
\underset{\ep, \eta, \ep (-log \eta)^\kappa \to 0}\limsup \ds \int_\R (v\chi_\eta) (x,t) dx \leq  C t^{1/3},
\ee
and, hence,  the claim.

%
%
%

\end{proof}

\section{The convergence--The proof of Theorem~\ref{main}}

The first part of the proof of Theorem~\ref{main} follows from the previous section. Indeed, the estimates there  imply that, along  subsequences, the pairs $(u_{\ep,\eta},\rho_{\ep, \eta})$ converge in the appropriate sense to solutions to \eqref{takis2} which satisfy the conditions for uniqueness. This yields the convergence to the unique viscosity solution to  \eqref{takis2}. However, although section~3 yields the local uniform convergence in $\R\times (0,1]$ of the $u_{\ep,\eta}$'s  to $\ol u$, the arguments  show that the $\rho_{\ep, \eta}$'s converge only in $L^2_{\text{loc}}(0,1); L^2(\R))$ and not in $L^2(\rt)$, which is necessary to prove the convergence of the functionals. 
\vs 
Here we show that this convergence is global by first
establishing that the $(u_{\ep,\eta}, \rho_{\ep,\eta})$'s are close to the $(\bar u_\eta, \bar \rho_\eta)$'s and then appealing to the convergence of the latter to the physics solution. Both limits are shown  in the strong topologies. This is done using a very important estimate, motivated by the theory of mean field games, that measures the difference between $(u_{\ep,\eta}, \rho_{\ep,\eta})$ and  $(\bar u_\eta, \bar \rho_\eta)$. Of course, the difficulty is to deal with  the contribution of the terms arising from the viscosity in \eqref{takis1}.

\begin{proof}[The proof of Theorem~\ref{main}] As already discussed above, the convergence, as $\ep,\eta, \ep(-\log\eta)^\kappa  \to 0$,  of the solution to \eqref{takis1} to the viscosity solution to \eqref{takis2} is an immediate consequence of the estimates in section~3 and the uniqueness, under the appropriate conditions, of the solutions to \eqref{takis2}. We leave it up to the reader to fill in the straightforward details. 
\vs

The proof of the convergence of the variational problems  is based on the comparison 
between the solution $(u_{\ep,\eta}, \rho_{\ep,\eta})$ to \eqref{takis1} and  the solution $(\ol u_{\eta}, \ol \rho_{\eta})$ 
of the approximate first-order system  \eqref{takis200.0} which provides a global $L^2(\rt)-$control of  $\rho_{\ep,\eta}-\ol \rho_{\eta}$.   Recall that we know that, as $\eta\to 0$, $\ol u_{\eta}\to \ol u$ locally  uniformly in $\R\times (0,1]$   and $\ol \rho_{\eta} \to \ol \rho$ in $L^2(\rt).$

\vs
In what follows, to keep the wording a bit less complicated, we will say that an estimate is uniform in $\ep$ and $\eta$ instead of an estimate is uniform in $\ep$ and $\eta$ satisfying \eqref{uc0}. 
\vs

The standard computation for mean field-type systems based on writing the system satisfied but the difference of the solutions, multiplying the first equation by the difference of the densities and the second by the difference of the potentials, and, finally, adding the two equations and performing a number of routine calculations yields 
\[\bs
&\dfrac{d}{dt} \ds\int_\R (u_{\ep,\eta}-\ol u_\eta)(\rho_{\ep,\eta}-\ol \rho_\eta)dx\\[1.2mm]
&= \ds\int_\R \Big( (\rho_{\ep,\eta}-\ol \rho_\eta )^2 + \dfrac{\rho_{\ep,\eta} + \ol \rho_\eta}{2}(\pt_xu_{\ep,\eta}-\pt_x \ol u_\eta)^2 \Big) dx  + \ep \ds \int_\R \big(  \pt_x^2 \rho (u_{\ep,\eta}-\ol u_\eta) - \pt_x^2 u(\rho_{\ep,\eta}-\ol \rho_\eta)\big) dx.
\end{split}
\]
\vs
Since $u_{\ep,\eta}$ is smooth,  after integrating by parts the previous relationship,  we find 
\[
\bs
& \dfrac{d}{dt} \ds\int_\R (u_{\ep,\eta}-\ol u_\eta)(\rho_{\ep,\eta}-\ol \rho_\eta)(x,t) dx\\[1.2mm]
 &= \ds\int_\R \Big( (\rho_{\ep,\eta}-\ol \rho_\eta )^2 + \dfrac{\rho_{\ep,\eta} + \ol \rho_\eta}{2}(\pt_xu_{\ep,\eta}-\pt_x \ol u_\eta)^2 \Big) dx  + \ep \ds \int_\R \big(  \pt_x^2 u_{\ep,\eta}  \ol \rho_\eta - \rho_{\ep,\eta} \pt_x^2 \ol u_\eta\big) dx. 
 \end{split}
 \]

Integrating 
the identity above over $t\in (0,1)$ and noting that, since both $u_{\ep, \eta}$ and $\ol u_\eta$ are continuous in $\R\times [0,1]$, the  integration at $t=0$ does not create any problem, we get 
\[
\bs
&\ds \int_\R (u_{\ep,\eta}-\ol u_\eta)(\rho_{\ep,\eta}-\ol \rho_\eta)(x,1) dx - \ds \int_\R (u_{\ep,\eta}-\ol u_\eta)(\rho_{\ep,\eta}-\ol \rho_\eta)(x,0) dx\\[1.2mm]
&= \ds \int_0^1\ds \int_\R \Big( (\rho_{\ep,\eta}-\ol \rho_\eta )^2 + \dfrac{\rho_{\ep,\eta} + \ol \rho_\eta}{2}(\pt_xu_{\ep,\eta}-\pt_x \ol u_\eta )^2 \Big) dx dt + \ep\ds \int_0^1\ds \int_\R \big(  \pt_x^2 u_{\ep,\eta}  \ol \rho_\eta - \rho_{\ep,\eta} \pt_x^2 \ol u_\eta\big) dx dt,
\end{split}
\]
and, thus, 
\be\label{takis302}
\bs
\ds \int_0^1\ds \int_\R& \Big( (\rho_{\ep,\eta}-\ol \rho_\eta )^2 + \dfrac{\rho_{\ep,\eta}+ \ol \rho_\eta}{2}(\pt_xu_{\ep,\eta}-\pt_x \ol u_\eta )^2 \Big) dx dt\\[1.2mm]
&=\ds \int_\R (u_{\ep,\eta}-\ol u_\eta)(\rho_{\ep,\eta}-\ol \rho_\eta)(x,1) dx - \ds \int_\R (u_{\ep,\eta}-\ol u_\eta)(\rho_{\ep,\eta}-\ol \rho_\eta)(x,0) dx \\[1.2mm]
&\hskip1in +  \ep\ds \int_0^1\ds \int_\R \big( \rho_{\ep,\eta} \pt_x^2 \ol u_\eta - \pt_x^2 u_{\ep,\eta}  \ol \rho_\eta \big) dx.
\end{split}
\ee
\vs 
The claim about the convergence of the $u_{\ep,\eta}$'s and $\rho_{\ep,\eta}$'s  will follow, if we show that  the limit,  as $\ep, \eta, \ep(-\log ({\eta})^\kappa \to 0$, of the right hand side of \eqref{takis302} is nonpositive.
\vs

First we observe that, since $\rho_{\ep,\eta}(\cdot,1)=\ol \rho_\eta(\cdot,1)=\text{\boldmath $ \delta$}$ and $u_{\ep.\eta}(x,0)=\ol u_\eta(x,0)=x^2/(2\eta)$, we have 
\be\label{takis303}
\ds \int_\R (u_{\ep,\eta}-\ol u_\eta)(\rho_{\ep,\eta}-\ol \rho_\eta)(x,1) dx- \ds \int_\R (u_{\ep,\eta}-\ol u_\eta)(\rho_{\ep,\eta}-\ol \rho_\eta)(x,0) dx=0.
\ee
\vs

Thus, to conclude we must show that 
\be\label{takis304}
\underset{\ep, \eta, \ep(-\log ({\eta})^\kappa \to 0}\limsup  \ep\ds \int_0^1\ds \int_\R \big( \rho_{\ep,\eta} \pt_x^2 \ol u_\eta - \pt_x^2 u_{\ep,\eta}  \ol \rho_\eta \big) dx\leq 0.
\ee

In view of the construction of $\ol u_\eta$, we know that $\pt_x^2 \ol u_\eta(x,t) \lesssim  1/(t+\eta)$. Since $\rho_{\ep,\eta} \geq 0$ and $\int_\R \rho_{\ep,\eta}(x,t) dx=1$, it follows that 
\be\label{takis305}
\ep\ds \int_0^1\ds \int_\R  \rho_{\ep,\eta} (x,t) \pt_x^2 \ol u_{\eta} (x,t) dx dt \lesssim  \ep \int_0^1 \dfrac{1}{t+\eta} dt=\ep \big[ \text{log} \big(1+\eta) - \text{log} \eta \big],
\ee
and, hence,
\be\label{ns1}
\underset{\ep, \eta, \ep(-\log ({\eta})^\kappa \to 0} \limsup \ep\ds \int_0^1\ds \int_\R  \rho_{\ep,\eta} (x,t) \pt_x^2 \ol u_{\eta} (x,t) dx dt \leq 0.
\ee

\vs
We turn now to the term $\ep \displaystyle \int_0^1\int_\R \pt_x^2 u_{\ep,\eta}(x,t) \ol \rho_\eta (x,t) dx dt$.  Its  analysis is more complicated consisting of splitting the time integration in integrals over $(0,h), (h,1-h)$ and $(1-h,1)$ and using different arguments for each one. 
\vs
We begin with the integral over $(1-h, 1)$ and recall that $\ol \rho_\eta \in L^2(\rt)$ with compact support. We claim that 
\be\label{ns2}
|\ep \ds\int_{1-h}^1\ds\int_\R \ol \rho_\eta \pt_x^2 u_{\ep,\eta} dx dt| \ \text{is small in $h$ uniformly in $\ep$ and $\eta$.}
\ee

To set up the proof of \eqref{ns2}, we recall that 
the uniform in $\ep$ and $\eta$ bound of the $u_{\ep,\eta}$'s\\
 in $ L^2((1/4, 1/2); H_{\text{loc}}^1({\R})$ yields some $t_{\ep,\eta} \in (1/4,1/2)$ such that 
\be\label{takis1100} \|\pt_x u_{\ep,\eta}(\cdot, t_{\ep,\eta})\|_{L^2(\R)}^2\lesssim  1.\ee

 We present next   a formal argument, which ignores the fact that some of the integrals over $\R$ may not be finite, to explain the idea behind \eqref{ns2}. The actual proof,  which will follow after the heuristics, relies on the use of appropriate smooth cut off functions. 
\vs


Multiplying the equation  satisfied by $u_{\ep,\eta}$   in \eqref{takis1} by $-\ep \pt_x^2 u_{\ep,\eta}$ and integrating over $\R\times (t_{\ep,\eta},1)$ leads to
\be\label{ns5}                                                                                                                                              
\bs
& \frac{\ep}{2} \ds \int_\R \big(\pt_x u_{\ep,\eta})^2(x,1) dx + \ep \|\pt^2_x u_{\ep,\eta}\|^2_{L^2(\R \times (t_{\ep,\eta} ,1))}\\
& \leq \|\rho_{\ep,\eta}\|_{L^2} \ep \|\pt_x u_{\ep,\eta}\|^2_{L^2(\R \times (t_{\ep,\eta} ,1))}  + \frac{\ep}{2} \ds \int_\R \big(\pt_x u_{\ep,\eta})^2(x,t_{\ep,\eta} ) dx.
\end{split}
\ee

It follows that 
\be\label{ns6}
\ep \|\pt^2_x u_{\ep,\eta}\|^2_{L^2(\R \times (t_{\ep,\eta},1))} \lesssim 1.
\ee
Thus

\be\label{ns7}
|\ep \ds\int_{1-h}^1\ds\int_\R \ol \rho_\eta \pt_x^2 u_{\ep,\eta} dx dt| \leq \Big(\ep \|\pt_x u_{\ep,\eta}\|^2_{L^2(\R \times (1-h,1))}\Big)^{1/2}   \Big(\ds \int_{1-h}^h \ds \int_\R (\ol \rho_\eta)^2 dx dt\Big)^{1/2},
\ee
and, in view of \eqref{ns6} and the fact that $\ol \rho_\eta \in L^2(\rt)$, there exists, a  uniform in $\ep$ and $\eta$,  $a:(0,1)\to \R$ such that  $\lim_{h\to 0} a(h)=0$ and
\be\label{ns7}
|\ep \ds\int_{1-h}^1\ds\int_\R \ol \rho_\eta \pt_x^2 u_{\ep,\eta} dx dt| \lesssim a(h).
\ee

We turn next to the rigorous proof of \eqref{ns2} which requires localization in space. For this, we use a cut-off  function 
\be\label{ua100} 
\phi=\chi^m \ \ \text{for some} \ m>1,
\ee
where 
\be\label{ns8}
\bc
\chi:\rt \to \R \ \ \text{is smooth, nonnegative and compactly  supported}, \\
 \chi\equiv 1 \  \ \text{in}  \   \underset{t  \in (0,1)} \cup \{(x,t): \; \ol \rho_\eta (x,t)>0 \}, \ \ \text{and}   \\
 |\pt_x \chi|\leq \chi^l \ \ \text{ for $l \in (0,1)$ such that   $m-1+l>0$}.                                                                                                                                      
\ec
\ee       
\vs
Multiplying the first equation in \eqref{takis1} by $-\ep \phi \pt_x^2 u_{\ep,\eta}$ and integrating over $\R$ yields 
\be\label{ua101}
\bs
\frac{\ep}{2}\frac{d}{dt} \ds\int_\R (\pt_x u_{\ep,\eta})^2\phi dx 
+\ep \frac{d}{dt} \ds\int_\R  u_{\ep,\eta} &\pt_x  u_{\ep,\eta} \pt_x \phi dx +\ep^2\ds\int_\R ( \pt_x^2 u_{\ep,\eta})^2\phi dx\\
& -\frac{\ep}{3} \ds \int_\R (\pt_x  u_{\ep,\eta})^3\pt_x \phi dx=-\ep \ds \int_\R \rho_{\ep,\eta} \pt_x^2 u_{\ep,\eta} \phi dx.
\end{split}
\ee
Noting that 
\[\frac{d}{dt} \ds\int_\R  u_{\ep,\eta} \pt_x  u_{\ep,\eta} \pt_x \phi dx = -\frac{\ep}{2}\frac{d}{dt} \ds\int_\R (u_{\ep,\eta})\pt_x^2\phi dx,\]
integrating \eqref{ua101} in time over the interval $(t_{\ep,\eta},1)$, and using \eqref{ns6}, \eqref{ua100} and the bound on $\|\rho_{\ep,\eta}\|^2_{L^2(\rt)}$, we  find
\be\label{ua102}
\bs
\frac{\ep^2}{2} \ds\int_{t_{\ep,\eta}}^1\ds \int_\R (\pt_x^2 u_{\ep,\eta})^2\chi^m dx & \lesssim  \Big(1+ \frac{\ep m}{3} \ds\int_{t_{\ep,\eta}}^1\ds \int_\R (\pt_x u_{\ep,\eta})^3 \chi^{m-1} \pt_x \chi dx \Big)\\
& \lesssim  \Big( 1+ \frac{\ep m}{3} \ds\int_{t_{\ep,\eta}}^1\ds \int_\R (\pt_x u_{\ep,\eta})^3 \chi^{m-1+l} dx\Big).
\end{split}
\ee                                                                                                               
To simplify the formulae we omitted in \eqref{ua102} the dependence of the $\lesssim$'s  on  previous bounds, which are independent of $\ep$ and $\eta$ as long as they satisfy \eqref{uc0}, and, finally,  $\chi$.
\vs 
Let
\[v_{\ep,\eta}=\big(\pt_x u_{\ep,\eta}\big) \chi^{\frac{m-1+l}{3}}.\]
It is immediate  from Proposition~\ref{prop125} that the $v_{\ep,\eta}$'s are bounded in  $L^2(\R\times (t_{\ep,\eta},1))$ uniformly on $\ep$ and $\eta$ satisfying  \eqref{uc0}. 

\vs
It then follows  from  an elementary Sobolev inequality that 
\be\label{ua103}
 \|v_{\ep,\eta}\|_{L^3(\R\times (t_{\ep,\eta},1))} \lesssim \Big [\|\pt_x v_{\ep,\eta}\|_{L^2(\R\times (t_{\ep,\eta},1))} +\|v_{\ep,\eta}\|_{L^2(\R\times (t_{\ep,\eta},1))}\Big ].
 \ee
A scaling argument implies that,  if $w$ satisfies
$$\|w\|_{L^3} \lesssim [|\pt_x w\|_{L^2} +\|v_{\ep,\eta}\|_{L^2}],$$
then 

$$L^{1/3}\|w\|_{L^3} \lesssim [L^{-{5/6}}\|\pt_x w\|_{L^2} + L^{1/6} \|v_{\ep,\eta}\|_{L^2}],$$

which implies, after minimizing over $L$, that

\[\|w\|_{L^3} \lesssim [|\pt_x w\|_{L^2} +\|v_{\ep,\eta}\|_{L^2}] \lesssim \|\pt_x w\|_2^{1/6} \|w\|_{L^2}^{5/6}.\]
\vs
Using this last estimate in \eqref{ua103} as well as Proposition~\ref{prop125}, we find that
\be\label{ua104}
\bs
\ep\ds\int_{t_{\ep,\eta}}^1 &\ds \int (\pt_x u_{\ep,\eta})^3 \chi^{{m-1+l}} dx dt \lesssim \ep \|\pt_x\big(\big(\pt_x u_{\ep,\eta}\big) \chi^{\frac{m-1+l}{3}}\big)\|_{L^2(\R\times (t_{\ep,\eta},1))}^{1/2}\\
&\lesssim \Big[\ep \Big(\ds\int_{t_{\ep,\eta}}^1\ds \int (\pt^2_x u_{\ep,\eta})^2 \chi^{\frac{2(m-1+l)}{3}} dx dt\Big)^{1/4} +
\ep  \Big(\ds\int_{t_{\ep,\eta}}^1\ds \int \pt_x \big(\chi^{\frac{m-1+l}{3}}\big) dx dt\Big)^{1/4} \Big].
\end{split}
\ee

Since, in view of \eqref{ns8},  $\ds\int_{t_{\ep,\eta}}^1\ds \int \pt_x \big(\chi^{\frac{m-1+l}{3}}\big) dx dt \lesssim 1$, it follows 
from \eqref{ua104} that 
\[\ep\|\pt_x^2 u_{\ep,\eta}\|_{L^2(\text{supp}(\ol \rho_\eta))} \lesssim 1,\]
and, thus,  \eqref{ns7} follows as above.
\vs

For the integral over $(h,1-h)$ we observe that, since $u_{\ep,\eta}$ is smooth and $\ol \rho_{\eta}$ is Lipschitz and has compact support,
\[
\bs
\ep\large \vert \ds \int_{h}^{1-h}\ds \int_\R \pt_x^2 u_{\ep,\eta}(x,t) \ol \rho_{\eta}(x,t) dx dt\big |&=\ep \big | \int_{h}^{1-h}\ds \int_\R \pt_x u_{\ep,\eta}(x,t) \pt_x\ol \rho_{\eta}(x,t)dx dt\big |\\[1.2mm]
& \lesssim  \ep   \Big( \ds \int_h^{1-h}  \|u_{\ep,\eta}(\cdot,t)\|^2_{H^1(\text{supp}(\ol \rho_\eta(\cdot,t)))} dt + \ds\int_h^{1-h} \ds\int_\R (\pt_x \ol \rho_{\eta})^2 dx dt\Big),
\end{split}
\]
and, hence, 
\be\label{takis307}
\underset{\ep\to 0}\lim \; \ep \ds \int_{1-h}^1\ds \int_\R \pt_x^2 u_{\ep,\eta}(x,t) \ol \rho_{\eta}(x,t) dx dt=0 \ \ \text{uniformly in $h\in (0,1/2)$}, 
\ee
since  Proposition~\ref{prop12} and Proposition~\ref{prop125} yield that,  for each $\theta\in (0,1/2)$,   $u_{\ep,\eta}$ and $\ol \rho_\eta$ are respectively bounded $L^2((\theta,1); H^1_{\text{loc}})$ and $L^2((\theta,1-\theta); H^1)$ for all $\ep,\eta$ satisfying 
\eqref{uc0}.
\vs

%

To conclude, we need to investigate the behavior of $\ep \ds \int_0^h \ds\int_\R \pt_x^2 u_{\ep,\eta} \ol \rho_{\eta} dx dt, $ and for this we recall an argument used earlier in the paper. 
\vs
If 
\[u_{\ep,\eta}(x,t)=v_{\ep,\eta}(x,t) + \dfrac{x^2}{2(t+\eta)} + \ep \text{log}\big(\dfrac{t+\eta}{\eta} \big),\]
then 
\[\pt_t v_{\ep,\eta} + \dfrac{1}{2} (\pt_x v_{\ep,\eta})^2 -\ep \pt_x^2 v_{\ep,\eta} + \dfrac{x}{t+\eta} \pt_x v_{\ep,\eta}=\rho_{\ep,\eta},  \ \   v_{\ep,\eta}(\cdot,0)=0 \ \ \text{and} \ \ v_{\ep,\eta}\geq 0.\]
Then,
\[\bs 
\ep\;\ds \int_0^h \ds \int_\R \pt_x^2 u_{\ep,\eta} \ol \rho_{\eta} dx dt&=\ep\; \ds \int_0^h \ds \int_\R \pt_x^2 v_{\ep,\eta} \ol \rho_{\eta} dx dt +
\ep\;\ds \int_0^h \ds \int_\R \dfrac{1}{t+\eta} \ol \rho_{\eta} dx dt\\[1.2mm]
&=\ep \; \ds \int_0^h \ds \int_\R \pt_x^2 v_{\ep,\eta} \ol \rho_{\eta} dx dt + \text{log} \big(\dfrac{t+\eta}{\eta}\big).
\end{split}
\]
To finish,  we need to estimate $\ep \; \ds \int_0^h \ds \int_\R \pt_x^2 v_{\ep,\eta} \ol \rho_{\eta} dx dt$.
\vs

Recalling that  $\ol \rho_{\eta}=r_\eta(t)\big(1-\dfrac{x^2}{l_\eta^2(t)}\big)_+$,  we find that 

\[\bs
\ep\ds \int_0^h \ds &\int_\R \pt_x^2 v_{\ep,\eta} \ol \rho_{\eta} dx dt = - \ep \; \ds \int_0^h \ds \int_\R \pt_x v_{\ep,\eta} \pt_x \ol \rho_{\eta} dx dt \\[1.2mm]
&= \ep \; \ds \int_0^h \big[ 2 v_{\ep,\eta}(l_\eta(t),t)\dfrac{r_\eta(t)}{l_\eta (t)} - 2  v_{\ep,\eta}(-l_\eta(t),t)\dfrac{r_\eta(t)}{l_\eta(t)} dt - 2\ep \; \ds \int_0^h \dfrac{r_\eta(t)}{l_\eta^2(t)} \ds \int_{\{|x|\leq l_\eta(t)\}} v_{\ep,\eta}(x,t) dx dt,
\end{split}
\]
and, hence, in view of the asymptotic behavior of the $r_\eta$'s and $l_\eta$'s near $t=0$, 
\[|\ep \;\ds \int_0^h \ds \int_\R \pt_x^2 v_{\ep,\eta} \ol \rho_\eta dx dt | \lesssim \ds \int_0^h \dfrac{1}{(t+\eta)^{4/3}} \|v_{\ep,\eta}(\cdot,t)\|_\oo dt.
\]
\vs 
Finally,  we need to estimate the right hand side of this last inequality, which have already done in the proof of Proposition~\ref{prop20.0}.  
\vs

Using Proposition~\ref{prop20.0} we find that 

\[|\ep \ds \int_0^h \ds \int_\R \pt_x^2 u_{\ep,\eta} \ol \rho_{\eta} dx dt| \leq C \ep \text{log}\big(\dfrac{h+\eta}{\eta}\big),\]

and, thus
\[\underset{\ep, \eta, \ep(-\log \eta)^\kappa \to 0} \lim |\ep \ds \int_0^h \ds \int_\R \pt_x^2 u_{\ep,\eta} \ol \rho_{\eta} dx dt|\lesssim 
\underset{\ep, \eta, \ep(-\log \eta)^\kappa \to 0} \lim  \ep \text{log}\big(\dfrac{h+\eta}{\eta}\big)=0.\]
\vs

The final item to show is the convergence of the $I_{\ep,\eta}$'s, for which we need to show that 
\be\label{ph1}
\underset{\eta\to 0} \lim \ds \int_0^1 \int_\R \big(\ol \rho_{\eta}-\ol \rho\big )^2 dx dt =0 \ \text{and} \  \underset{\eta\to 0} \lim \ds \int_0^1 \int_\R  \ol \rho_{\eta} (\pt_x \ol u_{\eta})^2 -\ol  \rho (\pt_x \ol u)^2 dxdt=0,
\ee
\be\label{ph100}
\underset{\ep, \eta, \ep(-\log \eta)^\kappa \to 0} \lim \ds \int_0^1 \int_\R \big(\rho_{\ep,\eta}-\ol \rho_\eta\big )^2 dx dt =0,
\ee
and
\be\label{ph130}
\underset{\ep, \eta, \ep(-\log \eta)^\kappa \to 0} \lim \ds \int_0^1 \int_\R  \big[\rho_{\ep,\eta} (\pt_x u_{\ep,\eta})^2 -\ol\rho_{\eta} (\pt_x \ol u_{\eta})^2 \big] dxdt=0.
\ee
\vs
The first set of limits, that is,   \eqref{ph1}, is  immediate from the explicit formulae for $\ol u_\eta, \ol \rho_\eta, \ol u$ and $\ol \rho$ in section~2.
\vs

The proofs of \eqref{ph100} and \eqref{ph130}, which are bit more complicated,  rely on \eqref{takis24},   \eqref{takis302} and the convergence of  the $(u_{\ep,\eta}, \rho_{\ep,\eta})$ to $(\ol u, \ol \rho)$ as $\ep,\eta, \ep(-\log \eta)^\kappa \to 0$ established in the first part of the proof.
\vs 
%
Indeed,  \eqref{takis303} and \eqref{takis304} yield that the $\ep,\eta, \ep(-\log \eta)^\kappa \to 0$ limit of the right hand side of 
 \eqref{takis302} is nonpositive.  Since the left hand side of  \eqref{takis302} is nonnegative, it follows that 
 \be\label{ph110}
\underset{\ep,\eta, \ep(-\log \eta)^\kappa \to 0} \lim \Big[ \ds \int_0^1\ds \int_\R \Big( (\rho_{\ep,\eta}-\ol \rho_\eta )^2 + \dfrac{\rho_{\ep,\eta}+ \ol \rho_\eta}{2}(\pt_xu_{\ep,\eta}-\pt_x \ol u_\eta )^2 \Big) dx dt\Big]=0,
 \ee
 and hence, in view of \eqref{ph1}, we get \eqref{ph100}.
 \vs
 
 To prove \eqref{ph130} we integrate \eqref{takis24} over the time interval $(0,1)$ to get 
 \be\label{ph140}
 \bs
 u_{\ep,\eta}(0,1)&= \rho_{\ep,\eta}(x,1)u_{\ep,\eta}(x,1)dx=  \ds \int_\R \rho_{\ep,\eta}(x,0) \dfrac{x^2}{2\eta} dx\\
 &+
 \ds \int_0^1 \ds \int_\R (\rho_{\ep,\eta})^2 dx dt + \dfrac{1}{2}  \ds \int_0^1 \ds \int_\R \rho_{\ep,\eta}(\pt_x u_{\ep,\eta})^2 dx dt = \dfrac{1}{2} \ds \int_0^1 \ds \int_\R (\rho_{\ep,\eta})^2 dx dt  + I_{\ep,\eta}.
 \end{split}
 \ee
As already mentioned earlier in the proof, in view of  Proposition~\ref{prop100}, we know that, in the limit $\ep\to 0,\eta \to 0$ and 
$\ep(-\log \eta)^\kappa \to 0$, the $u_{\ep,\eta}$'s, converge locally uniformly in $\R\times (0,1]$ to the  viscosity solution $\ol u$. Thus, 
 \be\label{ual1}
\underset{\ep,\eta,\ep(-\log \eta)^\kappa \to 0} \lim u_{\ep,\eta}(0,1) =\ol u(0,1). 
 \ee
Moreover,  \eqref{ph100} implies that 
 \be\label{ual2}
 \underset{\ep,\eta,\ep(-\log \eta)^\kappa \to 0} \lim \ds \int_0^1 \ds \int_\R  \rho_{\ep,\eta}^2 dx dt= \ds \int_0^1 \ds \int_\R \ol \rho^2 dx dt.
 \ee
 
 Finally, in view of Proposition~\ref{prop10}, we know that 
\be\label{ual3}
\underset{\ep,\eta,\ep(-\log \eta)^\kappa \to 0} \liminf I_{\ep,\eta} \geq \ol I= \ds \int_0^1 \ds \int_\R \dfrac{1}{2} \big [ \ol \rho^2 + \ol \rho (\pt_x \ol u)^2  \big] dx dt.
 \ee

 
 
 Using the observations above in \eqref{ph140} we get 
%
 \[\bs
 \ol u(0,1)&= \dfrac{1}{2} \ds \int_0^1 \ds \int_\R \ol \rho^2 dx dt  + \underset{\ep,\eta,\ep(-\log \eta)^\kappa \to 0} \limsup I_{\ep.\eta} 
  \geq \dfrac{1}{2} \ds \int_0^1 \ds \int_\R \ol \rho^2 dx dt  + \underset{\ep,\eta,\ep(-\log \eta)^\kappa \to 0} \liminf  I_{\ep,\eta}\\
 &  \geq \dfrac{1}{2} \ds \int_0^1 \ds \int_\R \ol \rho^2 dx dt  + \ol I
 = \ds \int_0^1 \dfrac{d}{dt} \Big [ \ds \int_\R \ol \rho \;\ol u dx\Big] dt \\
 &=\ol u(0,1).
 \end{split}
 \]
 
It follows that 
\[ \underset{\ep,\eta,\ep(-\log \eta)^\kappa \to 0} \lim I_{\ep,\eta}=\ol I.\]

We continue with the proofs of the last two claims of (i), namely, the convergence of  $\pt_x u_{\ep,\eta}$ and 
$\sqrt{\rho_{\ep,\eta}} \pt_x u_{\ep,\eta}$. 
\vs
Since all limits described here are as $\ep,\eta,\ep(-\log \eta)^\kappa \to 0$, in what follows to keep statements brief we will not repeating this fact.
\vs

Fix $\theta \in (0,1)$. Multiplying the equation satisfied by $u_{\ep,\eta}$ by a  smooth cut-off function $\phi$, which is compactly supported in $\R\times (\theta,1]$, we find, after integrating in space-time, that 
\[
\bs
\int_\theta^1\int_\R (\pt_x u_{\ep,\eta})^2 \phi dx dt &=2 \int_\theta^1\int_\R \rho_{\ep,\eta} \phi dx dt +2 \ep \int_\theta^1\int_\R u_{\ep,\eta}\pt^2_x\phi dx dt\\[1.2mm]
& - 2\int_\theta^1\int_\R u_{\ep,\eta} \pt_t \phi dx dt -2\int_R u_{\ep,\eta} (x,1)\phi(x,1) dx.
\end{split}
\]
In view of the local uniform convergence of the $u_{\ep,\eta}$'s to $\ol u$ and the $L^2$ convergence of the $\rho_{\ep,\eta}$ to $\ol \rho$, it follows that 
\[\bs 
\underset{\ep,\eta,\ep(-\log \eta)^\kappa \to 0}\lim \int_\theta^1\int_\R (\pt_x u_{\ep,\eta})^2 \phi dx dt&= 2 \int_\theta^1\int_\R \ol \rho \phi dx dt  - 2\int_\theta^1\int_\R \ol u \pt_t \phi dx dt -2\int_R \ol u(x,1)\phi(x,1) dx \\[1.2mm]
&= \int_\theta^1\int_\R (\pt_x \ol u)^2 \phi dx dt,\end{split}\]
with the second equality above following from the equation satisfied by $\ol u$.
\vs

Combining the above with the local in $\R\times (\theta, 1]$  weak convergence of the $\pt_x u_{\ep,\eta}$'s to $\pt_x \ol u$, which is a consequence of the estimates of the previous section, we get the claim about the local strong convergence in $\R\times (0,1] $ of the $\pt_x u_{\ep,\eta}$'s to $\pt_x \ol u$.
\vs

For the final claim,  we observe that the local $L^2$ convergence of the $\pt_x u_{\ep,\eta}$'s to $\pt_x \ol u$ and the $L^2$ convergence of the $\rho_{\ep,\eta}$'s to $\ol \rho$ imply that, almost everywhere  
in $\R\times (0,1)$,  $\pt_x u_{\ep,\eta} \to \pt_ x\ol u$, $\rho_{\ep,\eta} \to \ol \rho$, and, hence, $\sqrt{\rho_{\ep,\eta}} \pt_x u_{\ep,\eta} \to \sqrt{\ol \rho} \pt_x \ol u$.

\vs
 It also  follows from the uniform bound in $L^2$ of the $\sqrt{\rho_{\ep,\eta}} \pt_x u_{\ep,\eta}$'s, which is a consequence of the convergence of the $I_{\ep,\eta}$'s, that 
we also have that $\sqrt{\rho_{\ep,\eta}} \pt_x u_{\ep,\eta} \rightharpoonup \sqrt{\ol \rho} \pt_x\ol  u$. 
\vs

This observation together with the convergence of $\int_0^1\int_\R \rho_{\ep,\eta} (\pt_x u_{\ep.\eta})^2 dx dt$ to $\int_0^1\int_\R \ol \rho (\pt_x \ol u)^2 dx dt$, which is a consequence of the convergence of the $I_{\ep,\eta}$ to $\ol I$, yield the result.

\end{proof}

\begin{section}{The KPZ asymptotics and the proof of Theorem~\ref{main2}}

We begin with a short discussion about the background of the KPZ equation which we did not include, for the shake of brevity, in the introduction. Then we proceed with the proof of Theorem~\ref{main2}.

\subsection*{A short discussion about the KPZ equation} In full generality, the $1+1$ KPZ equation is 
\be\label{takis0}
\pt_t h -{\kappa}\partial_{x}^2 h + l (\partial_x h)^2 =\sqrt{D} \xi \ \ \text{in} \ \ \R\times (0,T),
\ee
where $\kappa$ is the diffusivity, 
$\xi=\xi(x,t)$ is a Gaussian space-time white noise with $\langle \xi(x,t) \rangle =0$ and $\langle \xi(x_1,t_1)\xi (x_2,t_2) \rangle =\delta (x_1-x_2)\delta(t_1-t_2)$ and $ l>0$ and $D>0$ are material parameters. The equation 
was introduced by Kardar, Parisi, and Zhang in   \cite{KPZ} to  model the dynamics of an
interface arising at the contact between a stable bulk phase with a metastable one. 
Assuming
that the two bulk phases have no conservation laws and relax exponentially fast, it was argued in \cite{KPZ}
 that the motion of the interface is governed by \eqref{takis0}, 
where $u=u(x, t)$ denotes the height function over the space $x\in\R$ at time $t>0$.
The nonlinearity arises from the asymmetry between the two phases, 
the second derivative  reflects the interface tension, and  the space-time white noise $\xi$ models
the randomness in transitions from metastable to stable. 
\vs
In the last twenty five years there has been  major progress in the mathematical theory of the KPZ equation and the properties of its solution. Listing all papers  is beyond the scope of this paper. Instead, 
We refer to the works by Ferrari and Spohn \cite{FS}, Quastel \cite{Q},  Corwin \cite{Co}, Quastel and Spohn \cite{QS}, Chandra and Weber \cite{CW}, Corwin and Shen \cite{CoS}, Hairer \cite{ Ha} and  Gubinelli,  Imkeller, and Perkowski \cite{GIP} and the references therein for   the mathematical study
of and issues related to the KPZ equation.
\vs

A topic that has been  the focus  of several works recently  is the understanding of the 
one-point probability distribution $P(H,t)$ of height $H$ at a specified point in space and time and the LDP for \eqref{takis0}. 
Several papers have produced  exact representations of $P(H,t)$  for arbitrary times. This 
progress has been achieved for three classes of initial conditions
(and some of their combinations and variations), namely flat
interface in Calabrese and Le Doussal \cite{CLD}, sharp or narrow wedge in Corwin \cite{Co}, Sasamoto and Spohn \cite{SS}, Calabrese, Le Doussal and Rosso \cite{CLDR}, Dotsenko \cite{Do} and Amir, Corwin and Quastel  \cite{ACQ}, and stationary interface, that is a two-sided Brownian interface pinned at a point, in  Imamura and Samasoto \cite{IS} and Borodin, Corwin, Ferrari and Veto \cite{BCFV} .
\vs

Large deviations of the KPZ equation have been intensively studied in the mathematics and physics communities
in recent years. The results obtained so far give a fairly complete picture in the long time regime $t\to \oo$. 
For the narrow wedge initial data, 
physics
literature predicted that the one-point, lower-tail Large Deviation Principle (LDP) rate function should go through a
crossover from a cubic  to a 5/2
power [KLD18b]. 
Corwin and Ghosal \cite{CG20a, CG20b} derived rigorous
bounds on the one-point tail probabilities for the narrow wedge initial data and general initial data. 
The lower-tail rate function was 
derived in the physics literature; see, for example, Sasorov, Meerson and Prolhac \cite{SMP17}, Krajenbrink,  Le Doussal and Prolhac \cite{KLDP18}, Corwin, Ghosal, Krajenbrink, Le Doussal, and Tsai \cite{CGK+18}, Krajenbrink, Le Doussal  and Prolh \cite{KLDP18} and Le Doussal \cite{LD19} and was rigorously established by Tsai \cite{Tsa18} and  Cafasso and Chayes \cite{CC19}. 
Although these references used 
different approaches,  it turns out, as shown in Krajenbrink and Le Doussal \cite{KLD19} , that the methods in 
were closely related. 
The entire
rate function  for the upper tail was derived by  Le Doussal, Majumdar and Schehr  \cite{LDMS16} for 
narrow wedge initial data. Rigorous proofs were given in  Das and Tsai \cite{DT19} and Ghosal and Lin \cite{GL}.
\vs

As already mentioned in the introduction the original motivation of our work was to provide a complete and rigorous proof to the large deviations-type (LDP) approximation for the KPZ equation for large heights and sharp-wedge initial condition.
\vs

\vs 
\subsection*{The proof of Theorem~\ref{main2}} We present here the proof of the second theorem.
\begin{proof}[The proof of Theorem~\ref{main2}]
We begin with the observation, following from the explicit calculations in section~2, that if, for $c>0$,  $(\ol u_c,\ol \rho_c)$ is the viscosity solution to 
\be\label{ph149}
\bs
& \pt_t \ol u_c +\dfrac{1}{2}\big(\pt_x \ol u_c)^2=\ol \rho_c \ \ \text{in} \ \ \rt,\\
& \pt_t \ol \rho_c + \pt_x \big(\ol \rho_c (\pt_x \ol u_c)\big)=0  \ \ \text{in} \ \ \rt,\\
& \ol \rho_c(\cdot,0)=c \text{\boldmath $ \delta$} \ \ \text{and} \ \ \ol \rho_c(\cdot,0)= \text{\boldmath $ \delta$},
\end{split}
\ee
then
\be\label{ph150}
\lim_{c\to 0} \ol u_c(0,1)=0 \ \ \text{and} \ \ \lim_{c\to oo} \ol u_c(0,1)=\oo.
\ee
\vs
We also remark  that a straightforward modification of the estimates in section~3 and the proof of Theorem~\ref{main} yields that,
if $(u_{c,\ep,\eta}, \rho_{c,\ep,\eta})$ is the solution to 
\be\label{ph151}
\bs
&\pt_t u_{c,\ep,\eta} -\ep \pt^2_x u_{c,\ep,\eta} +\dfrac{1}{2}\big(\pt_x u_{c,\ep,\eta} \big)=\rho_{c,\ep,\eta} \ \text{in} \ \rt,\\
&\pt_t \rho_{c,\ep,\eta} + \ep \pt^2_x \rho_{c,\ep,\eta} + \pt_x\big(\rho_{c,\ep,\eta} (\pt_x u_{c,\ep,\eta})\big)=0  \ \text{in} \rt,\\
& u_{c,\ep,\eta}(x,0)=\dfrac{x^2}{2\eta} \ \text{and} \ \rho_{c,\ep,\eta}(\cdot,1)=c \text{\boldmath $ \delta$},
\end{split}
\ee
then, in the sense of Theorem~\ref{main}, 
\be\label{ph152}
\underset{\ep,\eta, \ep(-\log \eta)^\kappa \to 0} \lim u_{c,\ep,\eta}=\ol u_c \ \ \text{and} \ \ \underset{\ep,\eta, \ep(-\log \eta)^\kappa\to 0} \lim \rho_{c,\ep,\eta}=\rho_c.
\ee
\vs
The next observation is that $u_{c,\ep,\eta}(0,1)$ and $\ol u_c(0,1)$ are both increasing in $c$. This is a consequence of the classical mean field-type identity, already been used several times in the paper, which  yields that 
\be\label{ph153}
\bs
\big(u_{c_2,\ep,\eta}(0,1)-&u_{c_1,\ep,\eta}(0,1)\big)(c_2-c_1)=\ds \int_\R \big(u_{c_2,\ep,\eta}-u_{c_1, \ep,\eta}\big)(x,1) \big(\rho_{c_2,\ep,\eta}-\rho_{c_1,\ep,\eta}\big)(x,1)dx\\
& + \ds\int_0^1\ds \int_\R \Big[\big(\rho_{c_2,\ep,\eta}-\rho_{c_1,\ep,\eta}\big)^2 + \dfrac{\rho_{c_2,\ep,\eta}+\rho_{c_1,\ep,\eta}}{2} \big(\pt_x u_{c_2,\ep,\eta} -\pt_x u_{c_1,\ep,\eta}\big)^2\Big] \geq 0, 
\end{split} 
\ee
and
\be\label{ph154}
\bs
\big(\ol u_{c_2}(0,1)&-\ol u_{c_1}(0,1)\big)(c_2-c_1)=\ds \int_\R \big(\ol u_{c_2}-\ol u_{c_1}\big)(x,1) \big(\ol \rho_{c_2}-\ol \rho_{c_1}\big)(x,1) dx\\
& + \ds\int_0^1\ds \int_\R \Big[\big(\ol \rho_{c_2}-\ol \rho_{c_1}\big)^2 + \dfrac{\ol \rho_{c_2}+\ol \rho_{c_2}}{2} \big(\pt_x \ol u_{c_2} -\pt_x \ol u_{c_1}\big)\Big] \geq 0. 
\end{split} 
\ee
It follows that 
\[\big(u_{c_2,\ep,\eta}(0,1)-u_{c_1,\ep,\eta}(0,1)\big)(c_2-c_1)\geq 0 \ \ \text{and} \ \  \big(\ol u_{c_2}(0,1)-\ol u_{c_1}(0,1)\big)(c_2-c_1)\geq 0,\]
and, hence, the claim.
\vs

Assume that, along a subsequence which we still denote the same way as the sequence, 
\[\underset{\ep,\eta, \ep(-\log \eta)^\kappa \to 0} \lim \ep^{3/2} c(\ep, \eta)=\oo.\]

Then, given $c>0$ and for $\ep$ and $\eta$ small enough,  we have $\ep^{3/2} c(\ep, \eta)>c$.  
Applying \eqref{ph153} to $(\hat u_{\ep,\eta}, \hat \rho_{\ep, \eta})$ and $( u_{c,\ep,\eta}, \rho_{c, \ep, \eta})$ and using that $\hat u_{\ep,\eta}(0,1)=1$,  we get 
\[ (1-u_{c,\ep,\eta}(0,1)) (\ep^{3/2} c(\ep, \eta) -c) \geq 0,\]
and, thus, 
\[1\geq u_{c,\ep,\eta}(0,1),\]
 which, in view of the fact that $\underset{\ep,\eta, \ep(-\log \eta)^\kappa \to 0} \lim u_{c,\ep,\eta}(0,1)=\ol u_c(0,1)$, contradicts the strictly increasing property of the $\ol u_c(0,1)$'s.

\vs

A similar argument implies that it is impossible for the $\ep^{3/2} c(\ep, \eta)$'s to converge along subsequences to $0$, and, thus, we have that 
\be\label{ph155}
\ep^{3/2} c(\ep, \eta) \ \text{is bounded away from $0$ and infinity uniformly as} \; {\ep,\eta, \ep(-\log \eta)^\kappa \to 0}.
\ee

Assume next that, along subsequences $\ep,\eta, \ep(-\log \eta)^\kappa \to 0$, the $\ep^{3/2} c(\ep, \eta) $'s converge to two different limits  
$c_1, c_2$ and, for definiteness,  assume that $c_1<c_2$. Along the same subsequences, Theorem~\ref{main} yields that $(\hat u_{c_1,\ep,\eta}, \hat \rho_{c_1,\ep,\eta})$ and $(\hat u_{c_2,\ep,\eta}, \hat \rho_{c_2,\ep,\eta})$ will converge respectively to  $(\ol u_{c_1},\ol \rho_{c_1})$ and $(\ol u_{c_2},\ol \rho_{c_2})$ and that  $\ol u_{c_1}(0,1)=\ol u_{c_2}(0,1)=1$. We also note that,   by construction  (see section~2), $\ol u_{c_1}(0,0)=\ol u_{c_2}(0,0)=0$. 
\vs
Then  \eqref{ph154} yields 
\be\label{ph160}
 \ds\int_0^1\ds \int_\R \Big[\big(\ol \rho_{c_2}-\ol \rho_{c_1}\big)^2 + \dfrac{\ol \rho_{c_2}+\ol \rho_{c_2}}{2} \big(\pt_x \ol u_{c_2} -\pt_x \ol u_{c_1}\big)\Big] dx dt =0,
\ee
and, hence,
\[\ol \rho_{c_2}=\ol \rho_{c_1}\]
which contradicts the fact that $c_1\neq c_2$.
\vs
Thus, there exists some $c_0\in (0,\oo)$ such that 
\[\underset{\ep,\eta, \ep(-\log \eta)^\kappa \to 0} \lim \ep^{3/2}c(\ep,\eta)=c_0.\]

The proof of the rest of the theorem follows from Theorem~\ref{main} and we leave it up to the reader to fill in the routine details.

\end{proof}

\vskip.075in

{\bf Statements and Declarations.} The authors state that there are no conflicts of interest. Data sharing is not applicable to this article as no datasets were generated
or analyzed during the current study.

\end{document}